\documentclass[11pt]{article}
\usepackage{graphicx}
\usepackage{epstopdf}
\usepackage{amsmath, amssymb}
\usepackage{latexsym, color}
\usepackage{mathtools}
\usepackage{xcolor}
\def\la{\big\langle}
\def\ra{\big\rangle}

\def\ds{\displaystyle}
\def\forall{\hbox{for all}~}
\def\L{{\bf L}}

\def\bfn{{\bf n}}
\def\bfe{{\bf e}}

\def\ve{\varepsilon}

\def\E{{\cal E}}
\def\A{{\cal A}}

\def\H{{\cal H}}\def\caL{{\cal L}} 
\def\dint{\int\!\!\int}

\def\R{{\mathbb R}}

\def\TV{\hbox{Tot.Var.}}

\def\implies{\Longrightarrow}
\def\vp{\varphi}

\def\F{{\cal F}}

\def\vs{\vskip 2em}

\def\v{\vskip 1em}
\def\O{{\cal O}}

\def\S{{\cal S}}
\def\C{{\cal C}}

\def\H{{\cal H}}

\def\bega{\begin{array}}
\def\enda{\end{array}}
\def\begi{\begin{itemize}}
\def\endi{\end{itemize}}
\def\ov{\overline}

\def\Tilde{\widetilde}
\def\Hat{\widehat}

\def\meas{\hbox{meas}}
\def\bel{\begin{equation}\label}
\def\eeq{\end{equation}}
\def\sqr#1#2{\vbox{\hrule height .#2pt
\hbox{\vrule width .#2pt height #1pt \kern #1pt
\vrule width .#2pt}\hrule height .#2pt }}
\def\square{\sqr74}
\def\endproof{\hphantom{MM}\hfill\llap{$\square$}\goodbreak}
\definecolor{cadmiumgreen}{rgb}{0.0, 0.42, 0.24}

\newtheorem{theorem}{Theorem}[section]
\newtheorem{corollary}{Corollary}[section]
\newtheorem{lemma}{Lemma}[section]

\newtheorem{remark}{Remark}[section]
\newtheorem{definition}{Definition}[section]

\begin{document}
\title{\bf Optimal Slicing Problems}
\vs
\author{Alberto Bressan$^{(1)}$, Maria Teresa Chiri$^{(2)}$ and Elsa M. Marchini$^{(3)}$
\\
\, \\
{\small $^{(1)}$Department of Mathematics, Penn State University,} {\small University Park, PA~16802, USA.}\\
 {\small $^{(2)}$~Department of Mathematics and  Statistics, Queen's University,
Kingston, ON K7L3N6,
Canada.}\\
{\small $^{(3)}$~Dipartimento di Matematica, Politecnico di Milano,} {\small Piazza L.\,da Vinci 32, Milano 20133, Italy.}\\
\, \\
{\small E-mails: axb62@psu.edu,~maria.chiri@queensu.ca,~elsa.marchini@polimi.it}
}
\maketitle

\begin{abstract}  Given a bounded open set $V\subset \R^2$, we consider the problem of slicing this set by a family of curves, whose average length is as small as possible.   The first sections of the paper prove the existence of an optimal slicing and establish a connection with the problem of eradicating an invasive biological species from an island,
considered in  \cite{BMS}. Namely, an optimal slicing can be obtained as a limit of minimum time eradication strategies.  The second part of the paper derives various necessary conditions for optimality,
describing the behavior of an optimal slicing strategy in the interior of $V$ and near the boundary $\partial V$.
\end{abstract}

\v
\section{Introduction}
\label{sec:1}
\setcounter{equation}{0}
Let $V\subset\R^2$ be a bounded open set, with 2-dimensional Lebesgue measure $\caL^2(V)=T$.
 By a {\bf slicing} of $V$ we mean a set-valued map $t\mapsto \Omega(t)\subseteq V$, 
 $t\in [0,T]$, such that 
 \bel{1} 0\leq t_1<t_2\leq T\qquad\implies\qquad \Omega(t_1)\subset\Omega(t_2),\eeq
 \bel{2} \caL^2\bigl(\Omega(t)\bigr)~=~t\qquad\qquad\forall~t\in [0,T].\eeq
We say that $ \Omega(\cdot)$ is an {\bf optimal slicing}  of the set $V$ if it minimizes the average length
of the relative boundaries. More precisely,
denoting by $\H^1$ the 1-dimensional Hausdorff measure~\cite{AFP, EG, M}, 
we consider
\begi\item[{\bf (OSP)}] {\bf Optimal Slicing Problem.} {\it Among all slicings $t\mapsto\Omega(t)$ of
the set $V$, find one that minimizes the functional
\bel{J}J(\Omega)~\doteq~\int_0^T \H^1\bigl(\partial \Omega(t)\cap V\bigr)\, dt.\eeq
}
\endi

An obvious way to establish the optimality of a slicing strategy is provided by
\begin{theorem}\label{t:11}
{\bf (a sufficient condition for optimality).} Let $V$ be bounded open set with measure $T=\caL^2(V)$,
and let
 $t\mapsto \Omega(t)$, $t\in [0,T]$ be a slicing of $V$.
   If
\bel{Dido}
\H^1\bigl(\partial\Omega(t)\cap V\bigr)\,=\,\min \Big\{ \H^1\bigl(\partial A\cap V\bigr)\,;~~A\subseteq V,~~
\caL^2(A)=t\Big\}\qquad\forall t\in [0,T],\eeq
then $\Omega(\cdot)$ is optimal.
\end{theorem}
In other words, if each set $\Omega(t)$ provides a solution to Dido's problem~\cite{Dido},  
minimizing the 
length of its relative boundary among all subsets with the same area, then $\Omega(\cdot)$ is optimal.  
This condition allows us to explicitly compute the optimal slicing of a disc, or an ellipse.
However, due to the additional monotonicity requirement (\ref{1}),
 a general convex set $V\subset \R^2$ may not admit any slicing which satisfies (\ref{Dido}).
In particular, Theorem~\ref{t:11} does not apply to any polygon.

The minimization problem for (\ref{J}) is closely related to a family of geometric evolution problems, modeling the spatial control 
of an invasive population 
\cite{BCS1, BCS2, BMS}. 
See \cite{BMN, BZ, CPo, CLP} for other control problems related to moving sets,
and \cite{ABM, BuBu} for more general geometric optimization problems.
 
For $t\in [0,T]$, we denote by $t\mapsto \Omega(t)\subseteq V$ a moving set. 
This can be regarded as a ``contaminated region", to be reduced as 
much as possible.   
We think of $V\subset\R^2$ as a geographical constraint, say, an island beyond which the invasive population cannot propagate.

To control the evolution of this set, we assign the 
velocity $\beta=\beta(t,x)$ in the inward normal direction at every point $x\in \partial \Omega(t)\cap V$
of the relative boundary.
A function $E(\beta)\geq 0$ is given, 
describing the {\bf effort} needed to push the boundary of $\Omega(t)$ inward, with speed $\beta$ in the 
normal direction (see Fig.~\ref{f:csm18}, left). 
 The {\bf total control effort} at time $t\in [0,T]$
is then defined as
\bel{Et} \E(t)~\doteq~
\int_{\partial\Omega(t)\cap V}E\bigl(\beta(t,x)\bigr)\, \H^1(dx),
\eeq
where the integral is computed w.r.t.~the 1-dimensional Hausdorff measure
along the relative boundary of $\Omega(t)$. 

Here we focus on the case where the effort function  is
\bel{E} E^\ve(\beta)~=~\max\bigl\{ 0, \ve+\beta\bigr\}\eeq
for some $\ve>0$.
Given a constant $M>0$ accounting for the maximum control effort, we consider set motions
$t\mapsto \Omega(t)$
which satisfy the constraint on the total control effort
\bel{EM} \E^\ve(t)~\doteq~
\int_{\partial\Omega(t)\cap V}E^\ve\bigl(\beta(t,x)\bigr)\, \H^1(dx)~
\leq~M\qquad\forall ~t\in [0,T].\eeq

\begin{figure}[ht]
\centerline{\hbox{\includegraphics[width=13cm]{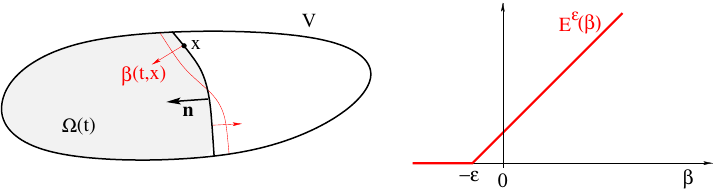}}}
\caption{\small Left: a moving set, where the evolution is determined by assigning the inward normal speed $\beta$
at each  point $x\in \partial\Omega(t)\cap V$ in the relative boundary.
Right: the effort function $E^\ve(\beta)$ in (\ref{E}).
}
\label{f:csm18}
\end{figure}

This models a situation where:
\begi
\item In absence of control, i.e.~if the control effort is everywhere zero: $E^\ve(\beta)=0$, then the inward normal speed is $\beta= -\ve$
at every point. 
Hence the contaminated set $\Omega(t)$ expands with speed $\ve$ in all directions.
In particular, its area increases at a rate proportional to the relative perimeter:
$${d\over dt} \caL^2\bigl(\Omega(t)\bigr)~=~- \int_{\partial \Omega(t)\cap V} \beta(t,x) \, \H^1(dx)~=~ \ve\,
\H^1\bigl(\partial \Omega(t)\cap V\bigr).$$
\item By implementing a control with total effort  $\E^\ve(t)= M$, we can clean up a region of area $M$ per unit time.   Assuming $\beta(t,x)\geq -\ve$ at every boundary point, the area of the contaminated set thus evolves according to:
$$\bega{l}\ds{d\over dt} \caL^2\bigl(\Omega(t)\bigr)~=\,-\int_{\partial \Omega(t)\cap V}  \beta(t,x) \, \H^1(dx)\\[4mm] \qquad =~\ds
\int_{\partial \Omega(t)\cap V} \Big[\ve - E^\ve\bigl(\beta(t,x)\bigr)\Big] \, \H^1(dx)~=~\ve\,\H^1\bigl(\partial \Omega(t)\cap V\bigr)-M.\enda$$
\endi
Having fixed the constants $M, \ve>0$, we say that 
 the set motion $t\mapsto\Omega(t)\subseteq V$ is {\bf admissible} if the control effort $\E^\ve(t)$ satisfies
 (\ref{EM}) at every time $t\in [0,T]$.
In the above setting, two problems can be formulated.

\begi
\item[{\bf (EP)}] {\bf Eradication Problem.}  {\it Given a bounded open set $V\subset\R^2$ and constants $M,\ve>0$, 
find an admissible set-valued function $t\mapsto \Omega(t)\subseteq V$  such that, for some $T>0$,
\bel{nco}\Omega(0)~=~V,\qquad\qquad 
\Omega(T)~=~\emptyset.\eeq
}
\endi

\begi
\item[{\bf (MTP)}]  {\bf Minimum Time Problem.}  {\it  Among all admissible strategies that satisfy
(\ref{nco}), find one which minimizes the time $T$.}
\endi
The analysis in \cite{BMS} has established the existence of optimal solutions,
together with necessary conditions and sufficient conditions for optimality.

In the present paper we begin by proving that, for any open set $V$ with finite perimeter,  there exists 
an optimal slicing (see Theorem~\ref{t:21}).

In Section~\ref{sec:3} we establish a connection between the two problems 
{\bf (OSP)} and {\bf (MTP)}.  Namely, let a bounded open set $V\subset\R^2$ be given, with 
$T\doteq \caL^2(V)$.   Fix $M=1$ and, for any $\ve>0$ small enough,
let $t\mapsto \Omega_\ve(t)$, $t\in [0, T_\ve]$ be a time optimal eradication strategy for the set $V$, 
with effort function  $E^\ve(\beta)\doteq\max\{0, \ve+\beta\}$.
By choosing a subsequence $(\ve_n)_{n\geq 1}$, we then prove the convergence $\Omega_{\ve_n}(t)\to \Omega^\sharp(t)$ for every $t\in [0,T]$, for some set-valued map $\Omega^\sharp$.  Moreover, after inverting time, 
we show that the map $t\mapsto \Omega^\sharp(T-t)$ provides an optimal slicing of the set $V$.

The remaining sections of this paper are concerned with  necessary conditions for the 
optimality of a slicing strategy.  
Assuming that the  the boundaries $\partial \Omega(t)$ admit a $\C^{1,1}$ parameterization
(continuously differentiable with Lipschitz derivative), in
Section~\ref{sec:4} we derive optimality conditions valid in the interior of $V$.
One should be aware, however, that for an optimal slicing strategy
the $\C^{1,1}$ regularity of the boundaries has been conjectured but not yet proved. 
As a step in this direction, in Section~\ref{sec:5} we show that 
the boundaries of the sets $\Omega(t)$ cannot have corners. Finally, in Section~\ref{sec:6}
we derive further optimality conditions at points $P(t)$ 
where the relative boundary of the moving set $\Omega(t)$ meets the boundary of $V$.

For particular geometric shapes, slicing strategies determined by the optimality conditions
are described in the forthcoming paper \cite{BCM3}.

\section{Existence of an optimal slicing strategy}
\label{sec:2}
\setcounter{equation}{0}
In this section we prove the existence of an optimal slicing.

\begin{theorem}
\label{t:21}
Every bounded open set $V\subset\R^2$  with finite perimeter admits a
slicing $t\mapsto \Omega(t)$ with minimum cost.
\end{theorem}

{\bf Proof.} {\bf 1.} We first show that a slicing with bounded cost exists.
Indeed, fix any unit vector $\bfe\in \R^2$ and  consider the sets
$$\Omega^\tau~\doteq~\bigl\{ x\in V\,;~~\langle \bfe, x\rangle <\tau\bigr\}.$$
Clearly,  the map $\tau\mapsto \psi(\tau)\doteq\caL^2(\Omega^\tau)$ is continuous and nondecreasing. 

By assumption, the set $V$ is bounded: $V\subseteq B(0,R)$ for some radius $R>0$.
This implies
$$0=\psi(-R)~<~\psi(R)~=~\caL^2(V)~\doteq ~T~<~\pi R^2.$$
Let $\tau\mapsto t(\tau)$  be a generalized inverse of $\psi$, i.e., an
increasing map such that $\psi(t(\tau))=t$ for every $t\in [0,T]$.
Then the set-valued function
$$t~\mapsto~\Omega^{\tau(t)},\qquad\qquad t\in [0,T],$$ is a slicing of the set $V$.
Moreover, each slice has measure $\H^1\bigl(V\cap\partial \Omega^{t(\tau)}\bigr)\leq 2R$.
Therefore, this slicing has cost $\leq 2R\, T$.
\v
{\bf 2.} Let   $m$ be the infimum among all slicing costs.
We can then construct a minimizing sequence of slicings, say $(\Omega_n)_{n\geq 1}\,$, with
\bel{minseq} m+1~>~J(\Omega_n)~\doteq~\int_0^T \H^1\bigl( V\cap \partial \Omega_n(t)\bigr)\, dt~~\to~~m.\eeq

We denote by ${\bf 1}_A$ the characteristic function of a set $A\subset \R^2$
By the properties (\ref{1})-(\ref{2}), the maps
$$t\mapsto {\bf 1}_{\Omega_n(t)}$$
from $[0,T]$ into $\L^2(\R^2)$ are all Lipschitz continuous with Lipschitz constant 1.

We would like to use Ascoli's theorem to extract a uniformly convergent subsequence.
However, this  is not straightforward, because the ranges of these maps may not be all contained
in a  compact set of $\L^1(\R^2)$.    The following steps cope with this problem.
\v
{\bf 3.} For $i\geq 1$, let $r_i\doteq 2^i$.   
We now modify each strategy $\Omega_n$, so that its values have uniformly bounded perimeter.

Observe that the map
$$t~\mapsto ~\psi_n(t)~\doteq~\H^1\bigl(V\cap\partial \Omega_n(t)\bigr) $$
is lower semicontinuous.  Hence its upper level sets are open, i.e., countable unions of open intervals. In particular,
$$\Big\{t\in \,]0,T[\,\,;~ \psi_n(t)>r_i\Big\}~=~\bigcup_j I_j\,,\qquad\qquad I_j=\,]a_j, b_j[\,.$$
Moreover, $$\psi_n(a_j)~\leq ~r_i,\qquad\quad \psi_n(b_j)~\leq ~r_i\,.$$
The modified  slicing strategy $t\mapsto \Omega_n^i(t)$ is defined as follows.   Fix a unit vector $\bfe\in\R^2$ and set
$$
\Omega_n^i(t)~\doteq~\Omega_n(t)\qquad \hbox{if} ~\psi(t)~\leq~r_i.
$$
On the other hand, if $t\in \,]a_j, b_j[\,$ for some $j$, we define $\Omega_n^i(t)$ as a suitable interpolation between the sets
$\Omega_n(a_j)$ and $\Omega_n(b_j)$.   More precisely (see Fig.~\ref{f:sli14})
\bel{omni}\Omega_n^i(t)~\doteq~\Omega_n(a_j)\cup \Big( \Omega_n(b_j)\cap \bigl\{ x\in\R^2\,;~~\langle \bfe, x\rangle \leq s_j(t)\bigr\}\Big),\eeq
where the increasing function $s_j(t)$ is determined by the area constraint 
$$\caL^2\bigl(\Omega_n^i(t)\bigr)~=~t.$$

The length of the relative perimeter of this set is estimated by
$$\bega{rl}
\H^1\bigl( V\cap \partial \Omega_n^i(t)\bigr)&\leq~\H^1\bigl( V\cap \partial \Omega_n(a_j)\bigr)+\H^1\bigl( V\cap \partial \Omega_n(a_j)\bigr)+\H^1\Big(\bigl\{  x\in V\,;~\langle \bfe, x\rangle= s_j(t)\bigr\}\Big)\\[3mm]
&\leq~r_i+r_i+\hbox{diam}(V)~\leq~2^{i+1} + 2R.\enda
$$
Since $V$ itself has bounded perimeter, this implies that all the sets $\Omega_n^i(t)$ have uniformly
bounded perimeter.   Moreover, by the integral bound in (\ref{minseq}) 
it follows
\bel{mess}\bega{rl}\meas\Big(\bigl\{ t\in [0,T]\,;~~\Omega_n^i(t)\not= \Omega_n(t)\bigr\}\Big)&=~\meas\Big(\bigl\{ t\in [0,T]\,;~~\H^1(V\cap \partial \Omega_n(t)) > r_i
\bigr\} \Big)\\[2mm]&\ds
<~{C\over  r_i}
~=~2^{-i}\, C.\enda\eeq

\begin{figure}[ht]
\centerline{\hbox{\includegraphics[width=9cm]{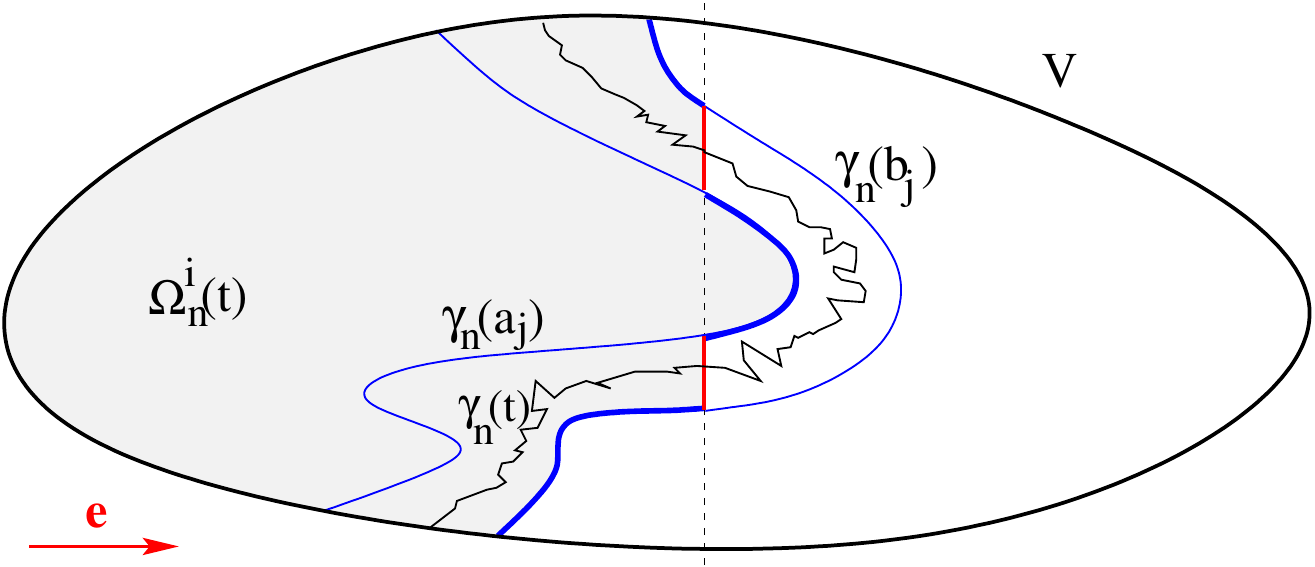}}}
\caption{\small The shaded region is the set $\Omega_n^i(t)$ defined at (\ref{omni}).  
Even if for some $t\in \,]a_j, b_j[\,$ the length
of the relative boundary $\gamma_n(t)\doteq V\cap \partial \Omega_n(t)$ is unbounded, all the curves $\gamma_n^i(t)\doteq V\cap \partial \Omega_n^i(t)$ have uniformly bounded length.
}
\label{f:sli14}
\end{figure}

\v
{\bf 4.}
 For every fixed $i\geq 1$, by construction all  the sets $\Omega^i_n$, $n\geq 1$, have uniformly bounded perimeter. Hence their characteristic functions lie in a compact subset of $\L^1(\R^2)$
 and we can thus apply Ascoli's theorem to the sequence $(\Omega_n^i)_{n\geq 1}$.
We proceed by induction on $i$.

 Let $n(1,p)$, $p\geq 1$, be a subsequence such that the limit exists, namely
 $$ \Big\| {\bf 1}_{\Omega^1_{n(1,p)}(t)}- {\bf 1}_{\Tilde \Omega_1(t)}\Big\|_{\L^1(\R^2)}~\leq~{1\over p} \qquad\forall t\in [0,T], ~~p\geq 1,$$
 for some  multifunction $t\mapsto \Tilde\Omega_1(t)$.
 
By induction, given the sequence $n(i-1,p)$, by Ascoli's theorem we can extract a further subsequence $n(i,p)$, $p\geq 1$, such that 
 $$ \Big\| {\bf 1}_{\Omega^i_{n(i,p)}(t)}- {\bf 1}_{\Tilde \Omega_i(t)}\Big\|_{\L^1(\R^2)}~\leq~{1\over p} \qquad\forall t\in [0,T], ~~p\geq 1.$$
 
 Now consider the diagonal sequence
 $\Omega^i_{n(i,i)}$.   By (\ref{mess}) we have
 $$\meas\Big(\bigl\{ t\in [0,T]\,;~~\Omega_{n(i,i)}^i(t)\not= \Omega_{n(i,i)}(t)\bigr\}\Big)
~\leq C\,2^{-i}.$$
 Since the series $\sum_i C\, 2^{-i}$ is bounded, by the Borel-Cantelli lemma it follows
 that there is a set of times ${\cal N}\subset [0,T]$ of measure zero such that
 $$t\notin{\cal N}\qquad\qquad \implies\qquad \Omega_{n(i,i)}^i(t)= \Omega_{n(i,i)}(t)$$
 for all except finitely many indices $i$.  
 
 Hence, for all $t\notin {\cal N}$, the limit 
 \bel{limoi}\lim_{i\to\infty} {\bf 1}_{\Omega_{n(i,i)}}(t)~=~{\bf 1}_{\Hat\Omega(t)}\eeq
 is well defined.   Since all functions $t\mapsto {\bf 1}_{\Omega_{n(i,i)}}(t)\in \L^1(\R^2) $ 
 are uniformly Lipschitz continuous with Lipschitz constant 1, the convergence holds uniformly 
 for all $t\in [0,T]$.
 
 Since each map $t\mapsto \Omega_{n(i,i)}(t)$ is a slicing of $V$, by the above 
 uniform convergence the limit map $t\mapsto \Hat\Omega(t)$ is a slicing as well.
 \v
 {\bf 6.} It remains to prove that $\Omega$ is optimal. 
 By the lower semicontinuity of the perimeter, from the limit (\ref{limoi}) it follows
 $$\H^1\bigl( V\cap\partial \Hat\Omega(t)\bigr)~\leq~\liminf_{i\to \infty} 
 \H^1\bigl( V\cap\partial \Omega_{n(i,i)}(t)\bigr)\qquad \quad\forall t\in [0,T].$$
Fatou's lemma thus yields
$$\int_0^T \H^1\bigl( V\cap\partial \Hat\Omega(t)\bigr)\, dt~\leq~\liminf_{i\to \infty}
\int_0^T \H^1\bigl( V\cap\partial \Omega_{n(i,i)}(t)\bigr)\, dt~=~m,$$
completing the proof.
\endproof

\section{Relations with time optimal eradication problems}
\label{sec:3}
\setcounter{equation}{0}

In the Introduction,  admissible motions were defined in terms of the constraints (\ref{Et})--(\ref{EM}).
Roughly speaking, the contaminated set $\Omega(t)$ expands with speed $\ve>0$ in all directions, while
the control removes a region of area $M$ per unit time.   Throughout the following,
to relate minimum time eradication problems with optimal slicing problems,
we fix the value $ M=1$.

In the regular case where the boundaries $\partial\Omega(t)$ admit a $\C^1$ parameterization,
the definitions of interior normal vector and of normal velocity $\beta$ in the inward direction are
clear.   However, in general the existence of optimal set motions can be achieved only within a family of sets with finite perimeter. In such case the normal vector is well defined only a.e.~w.r.t.~the Hausdorff measure.  As a preliminary, we recall here the precise definition of admissible motion, as introduced in \cite{BCS2}.

 Call $\F$ the family of all sets
$\Omega\subset [0,T]\times V\subset \R^3$ with finite perimeter.  Each $\Omega\in \F$ determines a 
set-motion
\bel{12}t~\mapsto~\Omega(t)~\doteq~\bigl\{ x\,;~~(t,x)\in \Omega\bigr\}.\eeq
For every point $(t,x)\in \partial^\star \Omega$ in the reduced boundary of $\Omega$ (see \cite{AFP, M} for 
a precise definition), 
let $\nu(t,x)= (\nu_0, \nu_1, \nu_2)\in \R^3$ be the (inward pointing) unit normal vector.
%

The (inward) normal velocity of the set 
$\Omega(t)$  at the point 
$(t,x)\in \partial^\star\Omega$ is then computed by
\bel{bdef}\beta~=~ {-\nu_0\over \sqrt{\nu_1^2 + \nu_2^2}}\,.\eeq
Choosing $E^\ve(\beta)\doteq\max\{ \ve+\beta, 0\}$,  the {\bf instantaneous total 
effort} is
$$\E^\ve(t)
~=~\int_{\partial\Omega(t)\cap V} E^\ve\left({-\nu_0\over \sqrt{\nu_1^2 + \nu_2^2}}\right)
\,\H^1(dx)
~=~\int_{\partial\Omega(t)\cap V} \max\left\{ {-\nu_0+ \ve\,\sqrt{\nu_1^2 + \nu_2^2}\over  \sqrt{\nu_1^2 + \nu_2^2}}\,,~0\right\}
\,\H^1(dx).$$
Therefore, integrating over a time interval $t\in \,]t_1,t_2[$ one finds
$$\int_{t_1}^{t_2} \E^\ve(t)\, dt ~=~\int_{\partial^*\Omega\cap\{(t,x);\, t_1<t<t_2\,,~x\in V\}} 
\max\Big\{ -\nu_0 +\ve\,\sqrt{\nu_1^2+\nu_2^2}\,,~0\Big\}\, d\H^2\,,$$
where $\H^2$ denotes the 2-dimensional Hausdorff measure.
To  impose the requirement that $\E^\ve(t)\leq 1$ for all $t$,    we consider the convex, positively homogeneous
function $L_\ve :\R^3\mapsto \R$, defined as
$$L_{\ve}(v) ~=~L_{\ve}(v_0,v_1,v_2)~\doteq~\max\Big\{ -v_0 +\ve\sqrt{v_1^2+v_2^2}\,,~0\Big\}.$$
In the following, $\H^2$ denotes the 2-dimensional Hausdorff measure.
\begin{definition}\label{d:13ve} Let  $V\subset\R^2$ be a bounded  open set with finite perimeter.
We say that  a set with finite perimeter $\Omega\in\F$ represents an  
$\ve$-{\bf admissible motion}, as in (\ref{12}), and write
$\Omega\in \A_{\ve}$, if for every $0\leq t_1<t_2\leq T$ one has
$$\int_{\partial^*\Omega\cap\{(t,x);\, t_1<t<t_2, \,x\in V\}} L_{\ve}\bigl(\nu(t,x)\bigr)\, d\H^2~\leq~t_2-t_1.$$
\end{definition}

\begin{lemma} \label{l:31}
For every $\ve>0$ small enough, the minimum time problem {\bf ($\ve$-MTP)} admits an optimal solution
$t\mapsto \Omega_\ve(t)$, $t\in [0, T_\ve]$.
\begi
\item[(i)]
As $\ve\to 0$ the minimum time $T_\ve$ decreases monotonically to $T=\caL^2(V)$.
\item[(ii)] By taking a suitable subsequence $\ve_n\to 0$,
one obtains the convergence 
\bel{cvg}\lim_{n\to \infty} \Big\| {\bf 1}_{\Omega_{\ve_n}} - {\bf 1}_{\Omega^\sharp}\Big\|_{\L^1\bigl(]0,T[\,\times\R^2\bigr)}~=~0,\eeq
for a set $\Omega^\sharp\subset [0,T]\times \R^2$ with finite perimeter.
 \endi
\end{lemma}

 {\bf Proof.} {\bf 1.}  For $\ve>0$ small enough, the eradication problem can be solved.  By the results in 
\cite{BMS} an optimal strategy $t\mapsto \Omega_\ve(t)$ exists.
The minimum time $T_\ve$ satisfies the identity
\bel{Tep}
T_\ve~=~\caL^2(V) + \ve \int_0^{T_\ve} \H^1\bigl( \partial \Omega_\ve(t)\cap V\bigr)\, dt.\eeq
Therefore, a strategy $t\mapsto\Omega_\ve(t)$ is optimal if and only if 
$$\int_0^{T_\ve} \H^1\bigl( \partial \Omega_\ve(t)\cap V\bigr)\, dt~=~\inf\left\{
\int_0^{T_\ve} \H^1\bigl( \partial \Omega(t)\cap V\bigr)\, dt~;\quad\Omega~\hbox{is $\ve$-admissible}\right\}.
$$
\v
{\bf 2.}  For every $\ve>0$ small enough, let $\Omega_\ve(\cdot)$ be an $\ve$-optimal eradication strategy and call $T_\ve$ the corresponding minimum time.
Notice that, if $0<\ve_1<\ve_2$, then $\Omega_{\ve_2}$ is also $\ve_1$-admissible.  This implies
$T_{\ve_1}\leq T_{\ve_2}$.   Hence by monotonicity there exists the limit:
$$T~=~\inf_{\ve>0} T_\ve~=~\lim_{\ve\to 0} T_\ve\,.$$
Another consequence of this monotonicity is that the integrals 
\bel{avp} \int_0^{T_\ve} \H^1\bigl( \partial \Omega_\ve(t)\cap V\bigr)\, dt\eeq
do not increase as $\ve\downarrow 0$.  In particular, they are uniformly bounded.
By (\ref{Tep}) we thus conclude
$$T~\doteq~\lim_{\ve\to 0} T_\ve~=~\caL^2(V),$$
completing the proof of (i).
\v
{\bf 3.} To prove (ii) we observe that, if $\Omega_\ve(\cdot)$ is a family of optimal strategies,
the integrals (\ref{avp}) are uniformly bounded.
This yields a uniform bound on the perimeters of the sets $\Omega_\ve$.

For every $\ve>0$, we write the reduced boundary as
$$
\partial^*\Omega_\ve \cap \bigl(]0,T_\ve[\times V\bigr)~=~\Sigma_{\ve}^1\cup\Sigma_{\ve}^2\cup\Sigma_{\ve}^3\,.
$$
Calling $\nu=(\nu_0, \nu_1,\nu_2)$ the unit inner normal vector at the point
$(t,x)\in \partial^*\Omega_{\ve}\cap V$, and defining 
the inner normal velocity $\beta_{\ve}=\beta_{\ve}(t,x)$ as in (\ref{bdef}), we consider 
the sets $\Sigma_{\ve}^i$, $i=1,2,3$,   defined as
$$\Sigma_{\ve}^1\,=\,\left\{(t,x)\in \partial^*\Omega_\ve \cap \bigl(]0,T_\ve[\times V\bigr) : |\nu_0(t,x)|\leq{1\over2}\right\},$$
$$\Sigma_{\ve}^2\,=\,\left\{(t,x)\in \partial^*\Omega_\ve \cap \bigl(]0,T_\ve[\times V\bigr) : \nu_0(t,x)<-{1\over2}\right\},$$
$$\Sigma_{\ve}^3\,=\,\left\{(t,x)\in \partial^*\Omega_\ve \cap \bigl(]0,T_\ve[\times V\bigr) : \nu_0(t,x)>{1\over2}\right\}.$$
The uniform bound of the integrals (\ref{avp}) yields
\bel{Sve1}
{\sqrt3\over2}\dint_{\Sigma_{\ve}^1}  \,d\H^2~\leq~\dint_{\Sigma_{\ve}^1}\sqrt{\nu_1^2+\nu_2^2} \,d\H^2
~\leq~ \int_0^{T_\ve} \H^1\bigl( \partial \Omega_\ve(t)\cap V\bigr)\, dt~\leq~ C_1 \eeq
for some constant $C_1$.
By the $\ve$-admissibility of $\Omega_{\ve}$, we obtain
\bel{Sve2}
 {1\over2}\dint_{\Sigma_{\ve}^2}\,d\H^2~\leq~
\dint_{\Sigma_{\ve}^2} L_{\ve}(\nu)\,d\H^2~\leq~\dint_{\partial^*\Omega_\ve \cap \bigl(]0,T_\ve[\times V\bigr)} L_{\ve}(\nu)\,d\H^2~\leq~ T_{\ve}~\leq ~C_2\eeq
for some constant $C_2$. 
On the other hand, 
$$\bega{l}
\ds-\caL^2(V)~=~\int_0^{T_{\ve}} \int_{\partial \Omega_{\ve}(t)} -\beta(t,x) \H^1(dx)\, dt   ~
=~\int_0^{T_{\ve}} \int_{\partial \Omega(t)} {\nu_0\over\sqrt{\nu_1^2+\nu_2^2}} \H^1(dx)\, dt  \\[4mm]
 \quad \ds
=~\int_{\Sigma_{\ve}^1\cup\Sigma_{\ve}^2\cup\Sigma_{\ve}^3} \nu_0\,d\H^2 
~ \geq~{1\over 2}  \int_{\Sigma_{\ve}^3}
d\H^2 - \int_{\Sigma_{\ve}^1\cup\Sigma_{\ve}^2} d\H^2\,,
\enda
$$
hence
$\H^2(\Sigma_{\ve}^3)~\leq~ 2\H^2(\Sigma_{\ve}^1\cup\Sigma_{\ve}^2)~\leq~ C_3$, for some constant $C_3$. 

Combining (\ref{Sve1}) and (\ref{Sve2}), we obtain a
uniform bound on the 2-dimensional measure $\H^2\Big(\partial^* \Omega_n \cap \bigl(]0,T[\times V\bigr)\Big)$
of the relative boundary of $\Omega_{\ve}$ inside $V$.  Since $V$ has  finite perimeter,  the characteristic functions ${\bf 1}_{\Omega_{\ve}}$ have uniformly 
bounded variation. 
A compactness argument (see Theorem~12.26 in \cite{M}) yields the existence of a subsequence
$\ve_n\to 0$ such that (\ref{cvg}) is satisfied, for some limit
set $\Omega^\sharp\in \F$ with finite perimeter.
\endproof
\v

Next, we wish to characterize the limit motion $t\mapsto\Omega^\sharp(t)$.

%

\begin{theorem}\label{t:31} 
Let $t\mapsto \Omega^\sharp(t)$, $t\in [0,T]$  be a limit motion obtained
as in (ii) of Lemma~\ref{l:31}.  Then 
$t\mapsto \Omega^\sharp(t)$ yields an optimal slicing of the set $V$
\end{theorem}

\begin{remark} {\rm   Taking the limit as $\ve\to 0$ in (\ref{E}), we obtain a totally degenerate 
minimum time problem.  Namely: minimize the time $T$ among all set motions $t\mapsto \Omega(t)$
such that
$$\meas \bigl(\Omega(t)\setminus\Omega(s)\bigr)~\leq~t-s\qquad\forall 0<s<t,$$
$$\Omega(0)=\emptyset,\qquad\Omega(T)=V.$$

In this trivial case, any set-valued map $t\mapsto \Omega(t)$ 
which is strictly increasing with $\caL^2(\Omega(t))= t$
is an optimal solution of {\bf (MTP)}.  However, this map will not provide an optimal solution
to the slicing problem, in general.
}\end{remark}

  {\bf Proof of the theorem.} {\bf 1.}
Let $\ve>0$.   Then for any $t\in [0, T_\ve]$ the optimal strategy $\Omega_\ve$ 
satisfies the analog of (\ref{Tep}), namely
\bel{tep}
\caL^2\bigl(\Omega_\ve(t)\bigr) ~=~t+ \ve \int_0^t \H^1\bigl( \partial \Omega_\ve(\tau)\cap V\bigr)\, d\tau.\eeq
W.l.o.g.~we can assume that 
\bel{omein}\Omega_\ve(t)~\subseteq~B\bigl(\Omega_\ve(s),~\ve(t-s)\bigr)\qquad\qquad \forall 0\leq s<t.\eeq
Indeed, if (\ref{omein}) fails for some $s$, one can consider the alternative eradication strategy
$$\Tilde \Omega_\ve(t)~\doteq~\left\{ 
\bega{cl} \Omega_\ve(t)\quad &\hbox{for}~~t\leq s,\\[1mm]
\Omega_\ve(t)\cap B\bigl(\Omega_\ve(s),~\ve(t-s)\bigr)\quad &\hbox{for}~~t> s.\enda\right.$$
This provides another admissible strategy with less or equal cost.

As in the proof of Lemma~\ref{l:31}, for any $t\in]0,T[$, by
possibly taking a subsequence, as $\ve\to 0$ we have
$$\Big\| {\bf 1}_{\Omega_{\ve}} - {\bf 1}_{\Omega^\sharp}\Big\|_{\L^1\bigl(]0,t[\,\times\R^2\bigr)}~\to~0\,.$$
Setting
$$\Omega^\sharp(t)\doteq\big\{x\in\R^2; (t,x)\in\Omega^\sharp\big\},$$
from (\ref{tep}) as $\ve\to 0$ one obtains 
\bel{meot}\caL^2(\Omega^\sharp(t))~=~t\quad \qquad\text{ for a.e. } t\in[0,T]\,.\eeq
Since $\Omega^\sharp(t)\subseteq V$ for every $t$,  (\ref{meot})  implies
\bel{lim0T} 
\lim_{t\to T-} \Big\| {\bf 1}_{\Omega^\sharp(t)} - {\bf 1}_V\Big\|_{\L^1}~=~0,
\qquad\quad \lim_{t\to 0+} \Big\| {\bf 1}_{\Omega^\sharp(t)} \Big\|_{\L^1}~=~0\eeq

Next, letting $\ve\to 0$ in (\ref{omein}), we obtain 
\bel{mont}\Omega^\sharp(s)~\subseteq~\Omega^\sharp(t)\qquad\qquad \forall 0\leq s<t.\eeq
Together, (\ref{meot})-(\ref{mont}) imply that $t\mapsto\Omega^\sharp(t)$ is a slicing of $V$.
\v
{\bf 2.} It remains to show that this slicing is optimal.   Assume not, and let $t\mapsto \Omega^*(t)$ be an optimal 
slicing of $V$, achieving a strictly lower value of the cost:
$$\int_0^T \H^1\bigl( \partial \Omega^*(t)\cap V\bigr)\, dt~<~\int_0^T \H^1\bigl( \partial \Omega^\sharp(t)\cap V\bigr)\, dt~\doteq~J^\sharp~\leq~\inf_{\ve > 0}~\int_0^{T_\ve} \H^1\bigl( \partial \Omega_\ve(t)\cap V\bigr)\, dt.$$
We will derive a contradiction, constructing an eradication strategy $\Omega_\ve(\cdot)$
with cost $J_\ve<J^\sharp$, for some $\ve>0$.

As a preliminary step, 
we claim that there exists a slicing strategy
$t\mapsto \Omega^\flat(t)$ such that
\bel{ssf} J^\flat~\doteq~ \int_0^T \H^1\bigl( \partial \Omega^\flat(t)\cap V\bigr)\, dt~<~J^\sharp,
\qquad \qquad \sup_{t\in [0,T]} \H^1\bigl( \partial \Omega^\flat(t)\cap V\bigr)~<~+\infty.\eeq
Indeed, if $$\sup_{t\in [0,T]} \H^1\bigl( \partial \Omega^*(t)\cap V\bigr)~<~+\infty$$
there is nothing to prove. Otherwise, we define $\gamma(t) ~\doteq~ \partial \Omega^*(t)$ and  
observe that the map
$$t~\mapsto ~\H^1\bigl(\gamma(t)\bigr) $$
is lower semicontinuous. We have
\bel{lrsm}
\lim_{r\to +\infty} \int_{\{\H^1(\gamma(t))>r\}}\H^1\bigl(\gamma(t)\bigr) \, dt ~=~0.\eeq
By lower semicontinuity, the domain of the above integral is an open set, i.e., a countable union of open intervals:
$$\Big\{t\in [0,T]\,;~ \H^1\bigl(\gamma(t)\bigr)>r\Big\}~=~\bigcup_j I_j\,,\qquad\qquad I_j=\,]a_j, b_j[\,.$$
Moreover, $$\H^1\bigl(\gamma(a_j)\bigr)~\leq ~r,\qquad\quad \H^1\bigl(\gamma(b_j)\bigr)~\leq ~r.$$

By (\ref{lrsm}) it also follows that for every constant $C$ one has
$$\lim_{r\to +\infty} (2r+C) \cdot \meas\Big(\Big\{t\in [0,T]\,;~ \H^1\bigl(\gamma(t)\bigr)>r\Big\}\Big)~=~0.$$
In particular, choosing 
$$C~=~\hbox{diam}(V)~=~\sup_{x,y\in V}~|x-y|$$ the diameter of the set $V$, we can choose $r$ large enough so that 
\bel{sl4}\bigl(2r+\hbox{diam}(V)\bigr)\cdot \meas\Big(\Big\{t\in [0,T]\,;~ \H^1\bigl(\gamma(t)\bigr)>r\Big\}\Big)~<~J^\sharp -J^*.\eeq
\v
{\bf 3.} 
The (sub-optimal) slicing strategy $t\mapsto \Omega^\flat(t)$ is defined as follows.   Fix a unit vector $\bfe\in\R^2$ and set
$$\Omega^\flat(t)~\doteq~\Omega^*(t)\qquad \hbox{if} ~\H^1\bigl(\gamma(t)\bigr)~\leq~r.$$
On the other hand, if $t\in \,]a_j, b_j[\,$ for some $j$, we define $\Omega^\flat(t)$ as a suitable interpolation between the sets
$\Omega^*(a_j)$ and $\Omega^*(b_j)$.   More precisely 
$$\Omega^\flat(t)~\doteq~\Omega^*(b_j)\cup \Big( \Omega^*(a_j)\cap \bigl\{ x\in\R^2\,;~~\langle \bfe, x\rangle \leq s_j(t)\bigr\},$$
where the increasing function $s_j(t)$ is uniquely determined by the area constraint 
$$\caL^2\bigl(\Omega^\flat(t)\bigr)~=~t.$$
The length of the relative perimeter of this set is estimated by
$$\bega{rl}
\H^1\bigl( \partial \Omega^\flat(t)\cap V\bigr)&\leq~\H^1\bigl( \partial \Omega^\flat(a_j)\cap V\bigr)+\H^1\bigl( \partial \Omega^\flat(a_j)\cap V\bigr)+\H^1\bigl( x\in V\,;~\langle \bfe, x\rangle \leq s_j(t)\bigr)\\[3mm]
&\leq~r+r+\hbox{diam}(V).\enda$$
In view of (\ref{sl4}), this shows that the slicing strategy $\Omega^\flat$ satisfies both inequalities in (\ref{ssf}).
%

\v
{\bf 4.} 
Consider the strategy
$$\Omega^\flat_\ve(t)~\doteq~\Omega^\flat\bigl(s_\ve(t)\bigr),$$
where
$$s_\ve(0)=0,\qquad\qquad {d\over dt} s_\ve(t)~=~1+{\ve\cdot \H^1\bigl( \partial \Omega^\flat(t)\cap V\bigr)}\,.
$$
By slightly ``slowing down"  the original strategy $\Omega^\flat$, we thus obtain a new strategy 
which is $\ve$-admissible.   
In the following we call $t_\ve(s)$ the inverse function, so that 
$${d\over ds} t_\ve(s)~=~\left[1+{\ve\cdot \H^1\bigl( \partial \Omega^\flat(t_\ve(s))\cap V\bigr)}\right]^{-1}$$

The total time needed to clean the set $V$ with this strategy satisfies
$$\bega{rl}T^\flat_\ve&\ds=~\caL^2(V) + \ve\int_0^{T_\ve^\flat} \H^1\bigl( \partial \Omega_\ve^\flat(t)\cap V\bigr)\, dt\\[4mm]
&\ds=~
\caL^2(V) + \ve\int_0^{T_\ve^\flat} \H^1\bigl( \partial \Omega^\flat\bigl(s_\ve(t)\bigr)\cap V\bigr)\, dt\\[4mm]
&\ds=~
\caL^2(V) + \ve\int_0^{T} \H^1\bigl( \partial \Omega^\flat(s)\cap V\bigr)\,{d t_\ve(s)\over ds} \,ds\,.
\enda
$$
Letting $\ve\to 0$, since the perimeters $\partial \Omega^\flat(s)\cap V$ have uniformly bounded length, we have the convergence
$$\int_0^{T} \H^1\bigl( \partial \Omega^\flat(s)\cap V\bigr)\,{d t_\ve(s)\over ds} ds\quad \to\quad \int_0^{T} \H^1\bigl( \partial \Omega^\flat(s)\cap V\bigr)\,ds~=~J^\flat~<~J^\sharp.$$
This shows that for $\ve>0$ small the strategies $\Omega_\ve$ are not optimal, reaching a contradiction.
\endproof

\section{Necessary conditions in the interior}
\label{sec:4}
\setcounter{equation}{0}

Let $t\mapsto \Omega(t)$, $t\in [0,T]$ be an optimal slicing of the set $V$.
In order to derive a set of necessary conditions for optimality, 
some regularity conditions  will be needed. Namely,
we shall assume that at least a portion of the 
boundaries $\partial \Omega(t)$ admits a $\C^{1,1}$ parameterization
(continuously differentiable with Lipschitz derivatives).

As shown in Fig.~\ref{f:csm3}, consider a $\C^{1,1}$ map
 $\Psi:(\tau,\xi)\mapsto x(\tau,\xi)$ with the following properties.
 \begi
\item[(i)]  {\it  The variables $(\tau,\xi)$ range over the  union of two rectangles:    
\bel{Wdef} W~\doteq~[\tau', \tau'']\times (I_1\cup I_2),\eeq
where $0<\tau_1<\tau_2< T$ and $I_1, I_2$ are disjoint closed intervals.}
 \item[(ii)] {\it 
At each time $t\in [\tau',\tau'']$, the map 
$$\xi~\mapsto ~x(t,\xi)~\in~\partial \Omega(t)\cap V$$
 is one-to-one.  Its range is contained in  the relative  boundary $\partial \Omega(t)\cap V$.}
 \item[(iii)]  {\it For every $(t,\xi)\in W$, the partial derivative
$ x_\xi(t,\xi)$
is a nonzero tangent vector to the boundary $\partial\Omega(t)$ at the point $x(t,\xi)$.}
\endi
By continuity and compactness, this implies
$$ \min_{(t,\xi)\in W} ~\bigl| x_\xi(t,\xi)\bigr|~>~0.$$
Therefore, by suitably choosing the orientation,  the perpendicular vector 
$$
\bfn(t,\xi)~\doteq~\left({x_\xi(t,\xi)\over \bigl| x_\xi(t,\xi)\bigr|}\right)^\perp$$
yields
the unit outer normal  to the set $ \Omega(t)$ at the boundary point $x(t,\xi)$.
\begi
\item[(iv)] {\it For each $\xi$,  the trajectory
$t\mapsto x(t,\xi)$ is orthogonal to the boundary $\partial\Omega(t)$ at every time $t\in [0,T]$.
Namely, there exists a  continuous scalar function $\beta:W\mapsto \R_+$ such that 
$$x_t(t,\xi)~=~\beta(t,\xi)\,\bfn(t,\xi)\qquad\qquad\forall (t,\xi)\in W.$$
}
\endi
We denote by 
\bel{curv}
\omega(t,\xi)~\doteq~{\la \bfn(t,\xi),\, x_{\xi\xi}(t,\xi)\ra\over \bigl|x_\xi(t,\xi)\bigr|^2}\eeq
the curvature of the boundary $\partial \Omega(t)$ at the point $x(t,\xi)$.
Notice that this curvature is well defined for a.e.~$(t,\xi)\in W$. Indeed, the functions
$x_\xi$ and $\bfn$ are locally Lipschitz continuous,  while $|x_\xi|$ is a continuous, strictly positive function.
By Rademacher's theorem~\cite{EG}, $x_{\xi\xi}$ exists almost everywhere.
The following regularity property will be assumed:
\begi
\item[(v)] {\it For each $\xi\in I_1\cup I_2$, the curvature function $t\mapsto \omega(t,\xi)$ 
is measurable and bounded. Moreover,
\bel{lepo} \lim_{\ve\to 0} ~\int_{\tau'}^{\tau''} \sup_{|\zeta-\xi|<\ve} \bigl|\omega(t,\zeta)-\omega(t,\xi)\bigr|\,dt~=~0.\eeq
}
 \endi

\begin{figure}[ht]
\centerline{\hbox{\includegraphics[width=15cm]{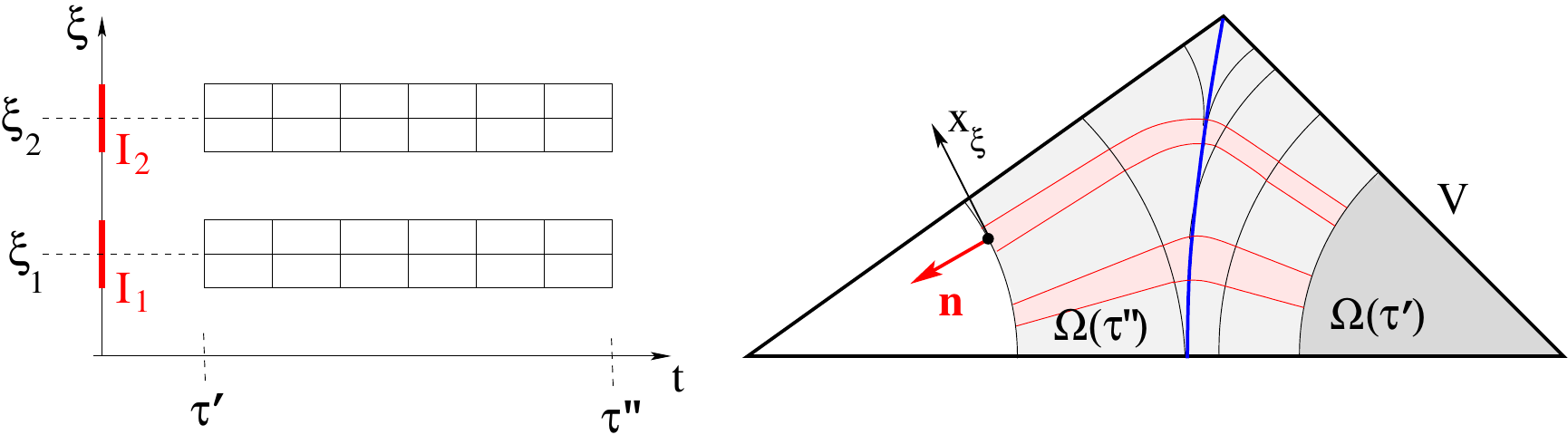}}}
\caption{\small Two tubes, parameterized by $(t,\xi)\mapsto x(t,\xi)$, where the point $(t,\xi)$
ranges over the two rectangles in (\ref{Wdef}).   The optimality of the slicing strategy is tested by 
enlarging the sets $\Omega(t)$ inside the first tube, and shrinking them by the same amount
inside the second tube.}
\label{f:csm3}
\end{figure}

Before stating a precise set of necessary conditions, we 
outline the main ideas.  
Let $t\mapsto \Omega(t)$ be an optimal slicing strategy.
Assume that there exist $\xi_1,\xi_2$, in the interior of the intervals $I_1,I_2$ respectively,  such that
$$\beta(\tau',\xi_1)~\doteq~\bigl| x_t(\tau',\xi_1)\bigr| ~>0,
\qquad\qquad   \beta(\tau'',\xi_2)~\doteq~\bigl| x_t(\tau'',\xi_2)\bigr| ~>0
.$$
This means: at the earlier time $\tau'$ the set $\Omega(\tau')$ is expanding in a neighborhood of 
the point $P_1= x(\tau', \xi_1) $, while
at the later time $\tau''$ the set $\Omega(\tau'')$ is expanding in a neighborhood of $P_2=x(\tau'', \xi_2)$.

For $\ve>0$ small, we construct a perturbed slicing strategy $t\mapsto  \Omega_\ve(t)$ such that
\begi 
\item   For all $t\notin [\tau',\tau'']$ one has $  \Omega_\ve(t)=  \Omega(t)$.
\item  For $t$ inside the time interval $ [\tau',\tau'']$, the difference $\Omega(t)\setminus\Omega_\ve(t)$
contains a small region with  area $A_\ve(t)$ in a neighborhood of $(t, \xi_1)$, while
 the difference $ \Omega_\ve(t)\setminus \Omega(t)$
contains a small region with the same area $A_\ve(t)$ in a neighborhood of $(t, \xi_2)$.
\endi

\begin{figure}[ht]
\centerline{\hbox{\includegraphics[width=5cm]{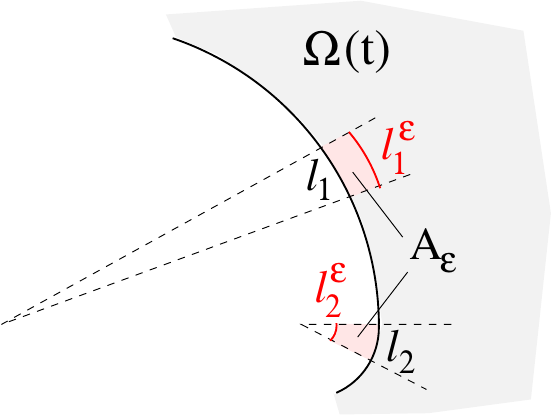}}}
\caption{\small  If the curvature of the boundary $\partial \Omega(t)$  near the point
$P_1=x(t,\xi_1)$ is smaller than near the point $P_2=x(t, \xi_2)$, by removing an area $A_\ve$
near $P_1$ and adding the same area near $P_2$,  according to (\ref{bdec}) the  perimeter decreases. 
}
\label{f:csm21}
\end{figure}

We now estimate the change in the total cost produced by the perturbation.
As shown in Fig.~\ref{f:csm21}, if an area $A_\ve$ is removed near a portion of boundary of length 
$l_1$ with curvature $\omega_1$, and added near  a portion of boundary of length 
$l_2$ with curvature $\omega_2$, to leading order the change in boundary length at time $t$ is
\bel{bdec}(l_1^\ve-l_1) + (l_2^\ve-l_2)~\approx~\omega_1(t,\xi_1) \cdot A_\ve -\omega_2(t,\xi_2) \cdot A_\ve\,.
\eeq
Hence, if
\bel{nec1}\int_{\tau'}^{\tau''} \omega_1(t,\xi_1) \, dt ~<~\int_{\tau'}^{\tau''} \omega_2(t,\xi_2) \, dt ,\eeq
we expect that the slicing strategy will not be optimal.

We can now state the main result of this section.

\begin{theorem}\label{t:41}
Let $t\mapsto\Omega(t)$, $t\in [0,T]$,  be an optimal slicing of the set $V\subset\R^2$. Let 
$\Psi: (t,\xi)\mapsto x(t,\xi)\in\partial\Omega(t)$ be a $\C^{1,1}$ map
with the properties (i)--(v) listed above.  Moreover, assume that for $t\in [\tau', \tau'']$ the boundary length $t\mapsto \H^1\bigl(\partial \Omega(t)\cap V\bigr)$ has bounded variation.

Consider any two points $\xi_1,\xi_2$ in the interior of $I_1,I_2$, respectively, and assume 
$\beta(\tau', \xi_1)>0$ and $\beta(\tau'',\xi_2)>0$. Then
\bel{intom}\int_{\tau'}^{\tau''} \omega(t,\xi_1) \, dt ~\geq~\int_{\tau'}^{\tau''} \omega(t,\xi_2) \, dt .\eeq
\end{theorem}

{\bf Proof.} {\bf 1.}   Assume that opposite inequality (\ref{nec1}) holds.   To derive a contradiction,
define the functions (see Fig.~\ref{f:sm76})
\bel{presc}\vp(s)~\doteq~\left\{ \bega{cl}{1\over 2} - s^2 &\hbox{for}~~|s|\leq {1\over 2},\\[1mm]
\bigl(1-|s|\bigr)^2&\hbox{for}~~ {1\over 2}\leq |s|\leq 1,\\[1mm]
0&\hbox{for}~~ |s|\geq 1,\enda\right.\qquad\qquad \vp_\ve(s)\doteq \vp\left( s\over\ve\right).\eeq
Notice that these are $\C^{1,1}$ functions, continuously differentiable with Lipschitz derivative and with compact support.
Moreover, they
satisfy
\bel{vpr1}
\int \vp(s)\, ds~=~{1\over 2}\,,\qquad\qquad \int \vp_\ve(s)\, ds~=~{\ve\over 2}\,,\eeq
$$\bigl[\vp'(s)\bigr]^2~\leq~4 \vp(s),\qquad \bigl[\vp_\ve'(s)\bigr]^2~\leq~4\ve^{-2} \vp_\ve(s)
\qquad\qquad\forall s\in \R.$$

\begin{figure}[ht]
\centerline{\hbox{\includegraphics[width=9cm]{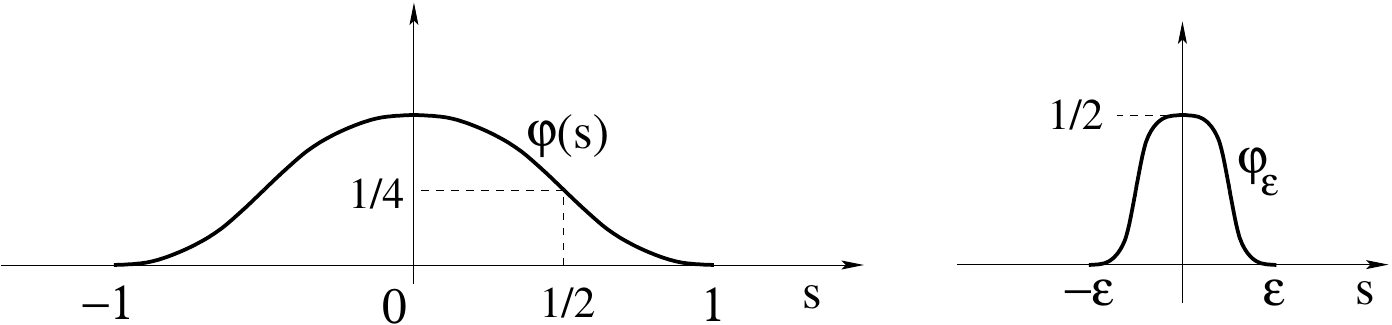}}}
\caption{\small The functions $\vp$ and $\vp_\ve$ introduced at (\ref{presc}). }
\label{f:sm76}
\end{figure}

{\bf 2.} To construct a family of perturbed strategies $t\mapsto\Omega_\ve(t)$ we proceed as follows.
Define the scalar functions $\sigma_{1,\ve}, \sigma_{2,\ve}$ by setting
 \bel{siep1}
\sigma_{1,\ve}(t)~\doteq~\left\{\bega{cl}\ds {-\ve^\gamma\over\bigl|x_\xi(t,\xi_1)\bigr|}\qquad
&\hbox{if} ~~t\in [\tau'+\ve,\tau''-\ve],
\\[4mm]
0\qquad &\hbox{if} ~~t\notin [\tau', \, \tau''],\enda\right.\eeq
\bel{siep2}
\sigma_{2,\ve}(t)~\doteq~\left\{\bega{cl}\ds {\ve^\gamma\over\bigl|x_\xi(t,\xi_2)\bigr|}\qquad
&\hbox{if} ~~t\in [\tau'+\ve,\tau''-\ve],
\\[4mm]
0\qquad &\hbox{if} ~~t\notin [\tau', \, \tau''].\enda\right.\eeq
On the remaining intervals $[\tau', \tau'+\ve]$ and $[\tau''-\ve,\tau'']$ the above functions  
are extended by a linear interpolation. Namely, for $i=1,2$ we define
\bel{suec}\left\{\bega{rl} \ds \sigma_{i,\ve}(t)~=~{t-\tau'
\over\ve}\cdot\sigma_{i,\ve}(\tau'+\ve)\qquad &\hbox{if}\quad  t\in [\tau',
 \,\tau'+\ve], \\[4mm]
 \ds
\sigma_{i,\ve}(t)~=~ {\tau''-t
\over\ve}\cdot \sigma_{i,\ve}(\tau''-\ve) \qquad &\hbox{if}\quad  t\in [\tau''-\ve, \tau''].
\enda\right.\eeq
 
 Next, recalling that $\bfn(t,\xi)$ is the outward unit normal vector to the boundary $\partial\Omega(t)$
at the point $x(t,\xi)$, for any $\ve>0$ sufficiently small we consider the sets $\Omega_\ve(t)$ whose boundaries are parameterized by
\bel{pert}
x^\ve(t,\xi)~=~\left\{ \bega{cl} x(t,\xi) \qquad\qquad \quad\qquad &\hbox{if}\quad 
|\xi-\xi_1|>\ve~~\hbox{or}~~|\xi-\xi_2|>\ve,\\[2mm]
x(t,\xi)+\sigma_{1,\ve}(t) \vp_\ve(\xi-\xi_1) \bfn(t,\xi)\qquad &\hbox{if}\quad 
|\xi-\xi_1|\leq\ve ,\\[2mm]
x(t,\xi)+\sigma_{2,\ve}(t) \vp_\ve(\xi-\xi_2) \bfn(t,\xi)\qquad &\hbox{if}\quad 
|\xi-\xi_2|\leq\ve .\enda\right.
\eeq
Roughly speaking, since $\sigma_{1,\ve}<0<\sigma_{2,\ve}$,  during the time interval $[\tau',\tau'']$ we are pushing the boundary $\partial \Omega(t)$  inward at points near 
$x(t,\xi_1)$ where the curvature is small, and outward at points $x(t,\xi_2)$ where the curvature is large.
Note that the portion of the boundary $\partial \Omega(t)$, which is not in the range of the parameterization map $\xi\mapsto \Psi(t,\xi)$, is left unchanged.

We claim that, if (\ref{nec1}) holds, this will decrease the average length of the boundaries, providing a contradiction.
\v
{\bf 3.} We now work out the key estimate.
Observing that $\langle x_\xi, \bfn\rangle=\langle \bfn_\xi, \bfn\rangle=0$, recalling the formula
(\ref{curv}) for the curvature, one obtains
\bel{cu2} \la \bfn_\xi(t,\xi),\, x_{\xi}(t,\xi)\ra ~=~-
\la \bfn(t,\xi),\, x_{\xi\xi}(t,\xi)\ra ~=~-\omega(t,\xi) \bigl|x_\xi(t,\xi)\bigr|^2.\eeq
For $|\xi-\xi_1|<\ve$, differentiating (\ref{pert}) w.r.t.~$\xi$ and then using (\ref{cu2}), we compute
\bel{xxi1}
x_\xi^\ve~=~x_\xi + \sigma_{1,\ve}\vp'_\ve \bfn +  \sigma_{1,\ve} \vp_\ve\bfn_\xi\,,\eeq
\bel{xxi2}\bega{rl}
|x_\xi^\ve|^2
&=~
|x_\xi|^2 + 2 \la x_\xi ,\,  \sigma_{1,\ve}\vp_\ve \bfn_\xi \ra+ \sigma_{1,\ve}^2 (\vp'_\ve)^2
+  \sigma_{1,\ve}^2\vp_\ve^2| \bfn_\xi|^2 \\[3mm]
&=~
\bigl(1-2\omega \sigma_{1,\ve} \vp_\ve\bigr) \,|x_\xi|^2 + \sigma_{1,\ve}^2 (\vp'_\ve)^2
+  \sigma_{1,\ve}^2\vp_\ve^2| \bfn_\xi|^2.
\enda\eeq
Taking square roots, since we are assuming that $|x_\xi|$ remains uniformly positive, we obtain
\bel{xxe}
|x^\ve_\xi| ~=~\bigl(1-\omega   \sigma_{1,\ve} \vp_\ve\bigr) \,|x_\xi| 
+ \O(1)\cdot   \sigma_{1,\ve}^2\bigl[(\vp'_\ve)^2 + \vp_\ve^2\bigr].\eeq
Here and throughout the sequel, the Landau symbol $\O(1)$ denotes a uniformly bounded quantity.
Similarly, for  $|\xi-\xi_2|<\ve$ one obtains
$$
|x^\ve_\xi| ~=~\bigl(1-\omega   \sigma_{2,\ve} \vp_\ve\bigr) \,|x_\xi| 
+ \O(1)\cdot   \sigma_{2,\ve}^2\bigl[(\vp'_\ve)^2 + \vp_\ve^2\bigr].
$$
Observing that 
$$\sigma_{i,\ve} =\O(1)\cdot \ve^\gamma, \qquad\vp_\ve= \O(1), \qquad\vp_\ve' = \O(1)\cdot \ve^{-1},$$
and recalling (\ref{vpr1}), we find
\bel{differ} \bega{l}\ds
\int_{\tau'}^{\tau''} \Big[ \H^1\bigl(\partial \Omega_\ve(t)\bigr) - \H^1\bigl(\partial \Omega(t)\bigr)\Big] \, dt
~=~\int_{\tau'}^{\tau''} \int_{I_1\cup I_2} \Big( \bigl| x_\xi^\ve(t,\xi)\bigr| -  \bigl| x_\xi(t,\xi)\bigr|
\Big)d\xi \, dt
\\[4mm]
\ds\qquad =~\int_{\tau'}^{\tau''}\sum_{i=1,2} \int_{I_i} \Big(-\omega   \sigma_{i,\ve} \vp_\ve\, \,|x_\xi| 
+ \O(1)\cdot   \sigma_{i,\ve}^2\bigl[(\vp'_\ve)^2 + \vp_\ve^2\bigr]\Big)d\xi~ dt
\\[4mm]
\qquad \ds =~\int_{\tau'}^{\tau''}\left(\int \ve^\gamma
\bigl[ \omega(t, \xi_1) \vp_\ve(\xi-\xi_1) - \omega(t,\xi_2)  \vp_\ve(\xi-\xi_2)+ \eta(t,\ve)\bigr] d\xi\right) dt
+\O(1)\cdot \ve\,\ve^{2\gamma} \, \ve^{-2}
\\[4mm]
\qquad \ds =~{\ve \over 2}\int_{\tau'}^{\tau''}\ve^\gamma
\bigl[ \omega(t, \xi_1)  - \omega(t,\xi_2) \bigr]dt  + \ve^{\gamma+1} \ov \eta(\ve)
+\O(1)\cdot \ve^{2\gamma-1}.
\enda\eeq
Here $\eta(t,\ve)$ is a term which keeps track of the differences
$$|x_\xi(t,\xi)| - |x_\xi(t,\xi_i)|,  \qquad\qquad \bigl| \omega(t,\xi)-\omega(t, \xi_i)\bigr| ,\qquad ~~i=1,2.$$
The assumption (\ref{lepo}) now implies $\ov\eta(\ve)\to 0$ as $\ve\to 0$.
By the Lipschitz continuity of $x_\xi$, we conclude that the right hand side of 
(\ref{differ}) is $<0$ for all $\ve>0$ sufficiently small,   provided we choose $\gamma$ so that $2\gamma-1 > \gamma+1$.
\v
{\bf 3.}
Based on the key estimate (\ref{differ}), to achieve a contradiction we  need to further
 modify the perturbed strategy (\ref{pert}), making sure it yields a slicing of the set $V$.  
 Indeed, the sets $\Omega_\ve(t)$  should still satisfy the identity
\bel{adm1}\caL^2\bigl(\Omega_\ve(t)\bigr)\,=\,t
\qquad\qquad\forall t\in [0,T],\eeq
together with the implication
\bel{adm2} \tau'\leq s<t\leq \tau''\qquad\implies\qquad \Omega_\ve(s)\subset \Omega_\ve(t).\eeq
In the next steps,  by a further modification of the sets $\Omega_\ve(t)$, 
we shall construct a new slicing $t\mapsto\Tilde \Omega_\ve(t)$ of $V$, with strictly lower cost.
\v
{\bf 4.} By (\ref{vpr1}) and (\ref{siep1})-(\ref{siep2}), for $t\in [\tau'+\ve, \tau''-\ve]$ one has
$$ -\int_{I_1} \bigl| x_\xi(t, \xi_2)\bigr| \sigma_{1,\ve}(t) \vp_\ve(\xi-\xi_1)\, d\xi~=~{\ve^{\gamma+1}\over 2}
~=~\int_{I_2} \bigl| x_\xi(t, \xi_2)\bigr| \sigma_{2,\ve}(t) \vp_\ve(\xi-\xi_2)\, d\xi.$$
One should keep in mind that $\Omega(t)$ is larger than $\Omega_\ve(t)$ in a neighborhood of 
$x(t, \xi_1)$, while
$\Omega_\ve(t)$ is larger than $\Omega(t)$ in a neighborhood of 
$x(t, \xi_2)$.
By the assumption that the tangent vector $x_\xi(t,\xi)$ is Lipschitz continuous, 
it follows
$$\caL^2\bigl( \Omega(t)\setminus \Omega_{\ve}(t)\bigr)~=~{\ve^{\gamma+1}\over 2}
+\O(1)\cdot \ve^{\gamma+2},$$
$$\caL^2\bigl( \Omega_{\ve}(t)\setminus \Omega(t)\bigr)~=~{\ve^{\gamma+1}\over 2}
+\O(1)\cdot \ve^{\gamma+2}.$$
We observe that the same estimates hold for all $t\in [\tau', \tau'']$, since the extensions
(\ref{suec}) are defined by convex interpolation.

The above analysis shows that the sets $\Omega_\ve(t)$ can violate the area identity (\ref{adm1}),
but only in the amount $\O(1)\cdot \ve^{\gamma+2}$.
 \v
 {\bf 5.} To understand by how much the monotonicity assumption (\ref{adm2}) is violated,
 consider the new outer normal:
$$\bfn^\ve(t,\xi)~\doteq ~{\bigl[x^\ve_\xi(t,\xi)\bigr]^\perp\over\bigl| x^\ve_\xi(t,\xi)\bigr|}\,.$$
Observe that the condition
\bel{inw}
 \langle x^\ve_t, \bfn^\ve\rangle~\geq~0\qquad\qquad\qquad\forall t,\xi\eeq
would  guarantee that all boundary points $x^\ve(t,\xi)\in\partial\Omega_\ve(t)$ move in the outward direction, thus implying (\ref{adm2}). We now check by how much (\ref{inw}) is violated.

For $|\xi-\xi_1|<\ve$, differentiating the second equation in (\ref{pert}) w.r.t.~time we  obtain
$$
x^\ve_t~=~x_t + \dot \sigma_{1,\ve} \vp_\ve \bfn + \sigma_{1,\ve}\vp_\ve \bfn_t\,.$$
Since $|\bfn^\ve|=|\bfn|=1$, by (\ref{xxi1}) we have
$$\bega{rl} |\bfn^\ve-\bfn|&=~\O(1)\cdot \la \bfn,\, (x^\ve_\xi-x_\xi)^\perp\ra
~=~\O(1)\cdot  \la x_\xi,\, (x^\ve_\xi-x_\xi)
\ra\\[2mm]
 &=~\O(1)\cdot \la x_\xi,\, \sigma_{1,\ve}\vp_\ve \bfn_\xi
\ra~=~\O(1)\cdot \sigma_{1,\ve}\vp_\ve\,.\enda
$$
Moreover, since $x_t$ is parallel to $\bfn$,
and $|x_\xi|$ remains uniformly positive, we have
\bel{nnep}\bega{rl}
 \langle x^\ve_t, \bfn^\ve-\bfn\rangle&=~\O(1)\cdot\bigl( |x_t| + | \dot\sigma_{1,\ve}|\vp_\ve \bigr)\cdot 
|\bfn^\ve-\bfn|^2  + \O(1)\cdot \sigma_{1,\ve} \vp_\ve \, |\bfn_t|\, |\bfn^\ve-\bfn|\\[3mm]
&=~\O(1)\cdot \sigma_{1,\ve}^2 \vp_\ve^2\,.
\enda
\eeq

Recalling (\ref{siep1}) and using the identity $\langle x_\xi,x_t\rangle \equiv 0$ together with (\ref{curv}), we now compute
\bel{dotsi1}\bega{rl} \dot\sigma_{1,\ve}(t)&\ds =~{-\ve^\gamma\over \bigl|x_\xi(t,\xi_1)\bigr|^2}\cdot {d\over dt} \bigl|x_\xi(t,\xi_1)\bigr|
~=~{-\ve^\gamma\langle x_{\xi}, x_{\xi t}\rangle\over \bigl|x_\xi(t,\xi_1)\bigr|^3}
~=~{\ve^\gamma\langle x_{\xi \xi}, x_t\rangle\over \bigl|x_\xi(t,\xi_1)\bigr|^3}\\[5mm]
&\ds =~{\ve^\gamma\langle x_{\xi \xi}, \bfn\rangle\over \bigl|x_\xi(t,\xi_1)\bigr|^2}\, {|x_t|\over |x_\xi|}
~=~\ve^\gamma{ \omega\over |x_\xi|}\,|x_t|~=~\O(1)\cdot \ve^\gamma \,|x_t|\,.
\enda
\eeq
An entirely similar computation yields
\bel{dotsi2} \dot\sigma_{2,\ve}(t)~=~\O(1)\cdot \ve^\gamma \,|x_t|\,.\eeq

The quantity in (\ref{inw}) can now be computed by 
\bel{xtn} \langle x^\ve_t, \bfn^\ve\rangle~=~ \langle x^\ve_t, \bfn\rangle 
+  \langle x^\ve_t, \bfn^\ve-\bfn\rangle~=~ \langle x_t, \bfn\rangle 
+ \dot\sigma_{i,\ve} \vp_\ve + \langle x^\ve_t, \bfn^\ve-\bfn\rangle.
\eeq
Here $i=1$ if $|\xi-\xi_1|<\ve$ and $i=2$ if $|\xi-\xi_2|<\ve$.   Note that if both of these alternatives fail,
then by construction $x^\ve(t,\xi)= x(t,\xi)$ and 
$$\langle x^\ve_t, \bfn^\ve\rangle~=~\langle x_t, \bfn\rangle~=~|x_t|~\geq ~0.$$

\begin{figure}[ht]
\centerline{\hbox{\includegraphics[width=9cm]{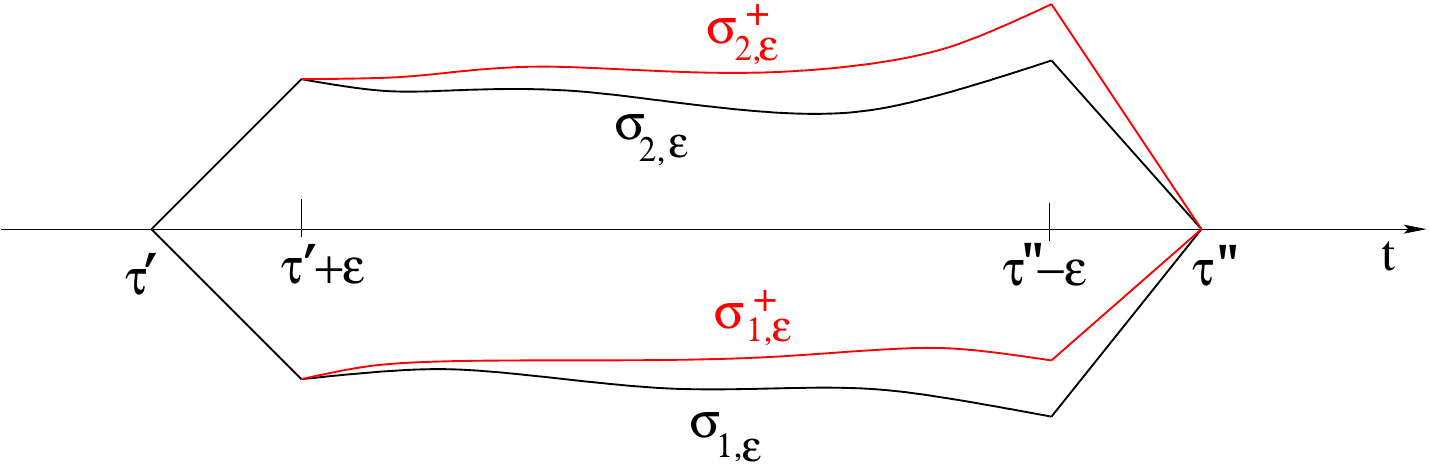}}}
\caption{\small The functions $\sigma_{i,\ve}$ and $\sigma_{i,\ve}^+$ introduced at 
(\ref{siep1})--(\ref{suec})
and at (\ref{si+})-(\ref{si++}). }
\label{f:csm22}
\end{figure}
\v
{\bf 6.} 
To achieve the monotonicity relations (\ref{adm2}) we define the modified functions (see Fig.~\ref{f:csm22})
\bel{si+}
\sigma_{i,\ve}^+(t)~\doteq~\left\{ \bega{cl} \sigma_{i,\ve}(t)\quad &\hbox{if} ~~t\notin [\tau'+\ve, \tau''],\\[1mm]
 \sigma_{i,\ve}(t)+ \ve^\alpha(t-\tau'-\ve) \quad &\hbox{if} ~~t\in [\tau'+\ve, \tau''-\ve],\enda\right.\qquad \quad i=1,2,\eeq
 for a suitable exponent $\alpha>0$, whose precise value will be chosen later.
We then extend both functions linearly on the remaining interval $[\tau''-\ve, \tau'']$ by setting
 \bel{si++}\sigma^+_{i,\ve}(t) ~=~ {\tau''-t\over\ve}\cdot \sigma_{i,\ve}^+(\tau''-\ve)\qquad
 \qquad t\in [\tau''-\ve, \tau''].\eeq

We claim that, replacing $\sigma_{i,\ve}$ with $\sigma^+_{i,\ve}$,
the inner product in (\ref{xtn}) is always $\ge 0$.
Notice that we only need to study the cases where $|\xi-\xi_1|<\ve$ or $|\xi-\xi_2|<\ve$,
because outside these intervals the boundaries of the sets $\Omega(t)$ are not changed. 

Three time intervals will be separately  considered:

CASE 1: $t\in [\tau', \tau'+\ve]$.   

Consider first  the case  where  $|\xi-\xi_1|\leq \ve$.
By assumption  $\bigl| x_t(\tau',\xi_1)\bigl|= \beta(\tau', \xi_1)>0$.   
By continuity, 
there exists $\delta_0>0$ such that
$$\bigl| x_t(t,\xi)\bigl|~>~\delta_0\,>\,0\qquad  \hbox{for} ~~t\in [\tau', \tau'+\ve], ~~|\xi-\xi_1|\leq\ve,$$
and all $\ve>0$ small enough.    Hence, again by continuity, the right hand side of (\ref{xtn}) remains
strictly positive for all $\ve>0$ small enough.

Next, consider the case where  $|\xi-\xi_2|\leq \ve$.  By (\ref{siep2})-(\ref{suec}) it follows 
(see Fig.~\ref{f:csm22})
$$\dot \sigma_{2,\ve}(t)~=~{\ve^{\gamma-1}\over \bigl| x_\xi(\tau'+\ve, \xi_2)\bigr|}~>~0\,,
\qquad \hbox{if} ~~t\in [\tau', \,\tau'+\ve],$$ 
\bel{xn5} \bega{rl} \langle x^\ve_t, \bfn^\ve\rangle
&=~ \langle x_t, \bfn\rangle + \dot\sigma_{2,\ve} \vp_\ve + \langle x^\ve_t, \bfn^\ve-\bfn\rangle
\\[2mm]
&\ds=~ |x_t|+{\ve^{\gamma-1}\over \bigl| x_\xi(\tau'+\ve, \xi_2)\bigr|}\vp_\ve + \O(1)\cdot \ve^{2\gamma} \vp_\ve^2.\\[2mm]
&\geq \ds~ 0
\enda\eeq
for all $\ve >0$ sufficiently small.

\v

CASE 2: $t\in [\tau''-\ve, \tau'']$.    

Assume first that  $|\xi-\xi_2|\leq \ve$. Since $\bigl| x_t(\tau'',\xi_2)\bigl|= \beta(\tau'', \xi_2)>0$,
by continuity
there exists $\delta_0>0$ such that
$$\bigl| x_t(t,\xi)\bigl|~>~\delta_0\,>\,0\qquad  \hbox{for} ~~t\in [\tau''-\ve, \, \tau''], ~~|\xi-\xi_2|\leq\ve,$$
and all $\ve>0$ small enough.    Hence, again by continuity, the right hand side of (\ref{xtn}) remains
strictly positive for all $\ve>0$ small enough.

Next, consider the case where  $|\xi-\xi_1|\leq \ve$.  By (\ref{siep2})-(\ref{suec}) it follows 
(see Fig.~\ref{f:csm22})
$$\dot \sigma^+_{1,\ve}(t)~=~-{1\over\ve} \sigma^+_{1,\ve}(\tau''-\ve)  ~=~{-1\over\ve}\left(
{-\ve^{\gamma}\over \bigl| x_\xi(\tau'+\ve, \xi_2)\bigr|} + \ve^\alpha (\tau''-\tau'-2\ve)\right).$$ 
Moreover, using (\ref{nnep}) with $\sigma_{1,\ve}$ replaced by $\sigma^+_{1,\ve}$, we obtain
\bel{xn6} \bega{rl} \langle x^\ve_t, \bfn^\ve\rangle
&=~ \langle x_t, \bfn\rangle + \dot\sigma_{1,\ve} \vp_\ve + \langle x^\ve_t, \bfn^\ve-\bfn\rangle
\\[2mm]
&\ds=~ |x_t|+{\ve^{\gamma-1}\over \bigl| x_\xi(\tau'+\ve, \xi_2)\bigr|}\vp_\ve + 
\O(1)\cdot \ve^{\alpha-1} \vp_\ve+ \O(1)\cdot (\ve^{\gamma} +\ve^\alpha)^2\vp_\ve^2\\[2mm]
&\geq ~0,
\enda\eeq
provided that $\alpha>\gamma$ and $\ve>0$ is suitably small.

\v
CASE 3:  $t\in [\tau'+\ve, \tau''-\ve]$.   By (\ref{si+}) and (\ref{dotsi1})-(\ref{dotsi2}), for $i=1,2$ it follows
\bel{plus} \dot \sigma^+_{i,\ve}(t)~=~\dot \sigma_{i,\ve}(t)+\ve^\alpha~=~\O(1)\cdot \ve^\gamma |x_t| + \ve^\alpha.\eeq
Replacing $\sigma_{i,\ve}$ with $\sigma_{i,\ve}^+$ in (\ref{xtn}) and using (\ref{nnep}), (\ref{plus}), we obtain
\bel{xtn+}\bega{rl} \langle x^\ve_t, \bfn^\ve\rangle&=~ \langle x_t, \bfn\rangle 
+ \dot\sigma_{i,\ve} \vp_\ve + \langle x^\ve_t, \bfn^\ve-\bfn\rangle\\[2mm]
&=~ |x_t|  +\O(1)\cdot \ve^\gamma |x_t| \vp_\ve +\ve^\alpha\vp_\ve
+ \O(1)\cdot (\sigma_{1,\ve}^+)^2 \vp_\ve^2\\[2mm]
&\geq~\ds {|x_t|\over 2} + \ve^\alpha \vp_\ve+\O(1)\cdot  (\ve^\alpha + \ve^\gamma)^2 \vp_\ve~\geq ~0
\enda\eeq
provided $\alpha< 2\gamma$ and $\ve>0$ is small enough.
\v
{\bf 7.} Call $t\mapsto \Hat\Omega_\ve(t)$ the set-valued map determined by the perturbation (\ref{pert}), 
replacing the functions $\sigma_{i,\ve}$ with $\sigma_{i,\ve}^+$, $i=1,2$.
By the previous step, if $\gamma>2$ and $\gamma<\alpha< 2\gamma$, the multifunction $\Hat\Omega_\ve$ satisfies the monotonicity assumption (\ref{adm2}).
However, a further modification is needed to satisfy (\ref{adm1}).  

By construction, we trivially have
$$\caL^2\bigl(\Hat\Omega_\ve(t)\bigr)~=~\caL^2\bigl(\Hat\Omega(t)\bigr)~=~t\qquad\qquad 
\forall t\notin [\tau', \tau''].$$
%
Moreover, the same argument used in step {\bf 4} now yields
$$\Big|\caL^2\bigl(\Hat \Omega_\ve(t)\bigr) - \caL^2\bigl(\Omega(t)\bigr)\Big|~=~
\O(1) \cdot \bigl( \ve^{\gamma+2}+ \ve^{\alpha+1}\bigr) .$$

As a consequence, there exists a unique Lipschitz continuous time rescaling
$t\mapsto t_\ve(t)$, implicitly defined by
$$\caL^2\bigl(\Hat\Omega_\ve(t_\ve(t)\bigr)~=~\caL^2\bigl(\Omega(t)\bigr)~=~t.$$
This map satisfies
\bel{tepp}t_\ve(t)-t~=~\left\{ \bega{cl}     0 \qquad &\hbox{if}\quad t\notin [\tau', \tau''],\\[2mm]
\O(1) \cdot \bigl( \ve^{\gamma+2}+ \ve^{\alpha+1}\bigr)  \qquad &\hbox{if}\quad t\in [\tau', \tau'']\,.
\enda\right.\eeq
By construction, the modified map
$$
t~\mapsto~\Tilde \Omega_\ve(t)~\doteq~\Hat\Omega_\ve\bigl(t_\ve(t)\bigr),$$
is a slicing of the set $V$.
\v
{\bf 8.} It remains to show that the cost of the new slicing $\Tilde\Omega_\ve(\cdot)$ 
is strictly smaller than
the cost of $\Omega(\cdot)$. 
For readers' convenience, we  collect all the assumptions  on the constants $\alpha,\gamma$:
\bel{ag}
\gamma\,>\,2,\qquad \gamma\,<\alpha\,<\,  2\gamma,\qquad \alpha\,>\,\gamma+1.\eeq
 
 We recall that, by (\ref{differ}), for all $\ve>0$ small enough, 
one has
\bel{dif2}
\int_{\tau'}^{\tau''} \Big[ \H^1\bigl(\partial \Omega_\ve(t)\bigr) - \H^1\bigl(\partial \Omega(t)\bigr)\Big] \, dt
~\leq~ {\ve^{\gamma+1}\over 3} \int_{\tau'}^{\tau''} \bigl[ \omega(t,\xi_1)-\omega(t,\xi_2)\bigr]\, dt~<~0\,.
\eeq
We need to check that a similar inequality still holds when $\Omega_\ve$ is replaced by $\Tilde\Omega_\ve$.

Repeating the estimate in (\ref{differ}) with $\sigma_{i,\ve}$ replaced by $\sigma^+_{i,\ve}$, $i=1,2$,
we observe that for every $t\in [\tau', \tau'']$ the integrand contains an additional term with size
$\O(1)\cdot \ve^\alpha \vp_\ve$.  For $\ve>0 $ small, the bound (\ref{dif2}) must therefore be replaced with
\bel{dif3}
\int_{\tau'}^{\tau''} \Big[ \H^1\bigl(\partial \Hat 
\Omega_\ve(t)\bigr) - \H^1\bigl(\partial \Omega(t)\bigr)\Big] \, dt
~\leq~ {\ve^{\gamma+1}\over 3} \int_{\tau'}^{\tau''} \bigl[ \omega(t,\xi_1)-\omega(t,\xi_2)\bigr]\, dt
+ \O(1)\cdot \ve^{\alpha+1}.
\eeq
Finally, by the assumption on the total variation of the boundary length $\H^1\bigl(\partial\Omega(t)\cap V\bigr)$,
 it follows that 
also the length of the perturbed boundaries 
$t\mapsto \H^1\bigl(\Hat \Omega_\ve\cap V\bigr)$ has uniformly bounded variation as a function of time
$t\in [\tau', \tau'']$.

Since the map $t\mapsto t_\ve(t)$ is increasing and satisfies (\ref{tepp}), applying Lemma~\ref{l:41} below
with $\delta\doteq \ve^{\gamma+2}$ we obtain
\bel{dif4}
\int_{\tau'}^{\tau''} \Big[ \H^1\bigl(\partial \Hat 
\Omega_\ve(t_\ve(t))\bigr) - \H^1\bigl(\partial \Hat\Omega(t)\bigr)\Big] \, dt
~=~\O(1)\cdot \ve^{\gamma+2}.
\eeq
Combining (\ref{dif3}) with (\ref{dif4}), by the choice of $\alpha$ at (\ref{ag}) we conclude
$$\bega{l}\ds
\int_{\tau'}^{\tau''} \Big[ \H^1\bigl(\partial \Tilde
\Omega_\ve(t)\bigr) - \H^1\bigl(\partial\Omega(t)\bigr)\Big] \, dt
\\[4mm]
\ds
\qquad\leq~ {\ve^{\gamma+1}\over 3} \int_{\tau'}^{\tau''} \bigl[ \omega(t,\xi_1)-\omega(t,\xi_2)\bigr]\, dt
+\O(1)\cdot \ve^{\gamma+2}~<~0\enda
$$
for all $\ve>0$ sufficiently small.
Hence $\Tilde \Omega_\ve(\cdot)$ yields a slicing of $V$ with strictly lower cost.
This contradiction establishes the theorem.
\endproof

\begin{figure}[ht]
\centerline{\hbox{\includegraphics[width=8cm]{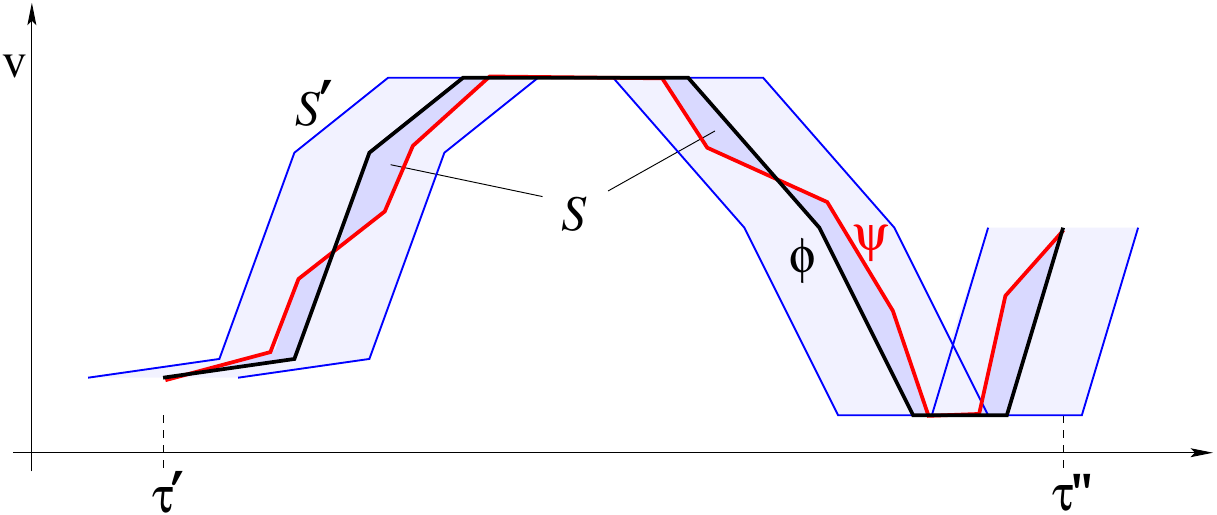}}}
\caption{\small Proving Lemma~\ref{l:41}.   Here $\S$ is the darker region bounded between the graphs of 
$\phi$ and $\psi$, while $\S'$ is the lightly shaded region.}
\label{f:csm23}
\end{figure}

\begin{lemma}\label{l:41}
Let $\phi:[\tau',\tau'']\mapsto\R$ be a function with bounded variation.
Let $t\mapsto \zeta(t)$ be an increasing  map of $[\tau',\tau'']$ onto itself with
$$\bigl|\zeta(t)-t\bigr|~\leq~\delta\qquad\forall ~t\in [\tau',\tau''].$$
Then 
$$
\int_{\tau'}^{\tau''} \bigl| \phi(\zeta(t))-\phi(t)\bigr|\, dt~\leq~2\delta\cdot \TV\{ \phi\}.
$$
\end{lemma}
{\bf Proof.} By an approximation argument, we can assume that $\phi$ is continuous.
Consider the continuous function $\psi(t)\doteq\phi(\zeta(t))$ and 
the sets (see Fig.~\ref{f:csm23})
$$\S~\doteq~\Big\{ (t,v)\,;\quad t\in [\tau',\tau''],~~v = \theta \psi(t) +(1-\theta) \phi(t),~~\theta\in [0,1]\Big\},$$
$$\S'~\doteq~\Big\{ (s,\phi(t))\,;~~t\in [\tau',\tau''],~~|s-t|\leq\delta\Big\}.$$
Since $\S\subseteq\S'$, we  have
$$\|\psi-\phi\|_{\L^1} ~=~\meas(\S)~\leq~\meas(\S')~\leq~2\delta\cdot \TV\{\phi\}.$$
\endproof
\v
\begin{corollary} In the same setting as Theorem~\ref{t:41}, 
 for a.e.~time $t\in [\tau', \tau'']$, the curvature of the boundary 
$\partial\Omega(t)$
is constant at a.e.~point $x(t,\xi)$ where it is moving, i.e.~where $\beta(t,\xi)>0$.
\end{corollary}

{\bf Proof.} {\bf 1.} By the assumed $\C^{1,1}$ regularity of the map $(t,\xi)\mapsto x(t,\xi)$,
the map $(t,\xi)\mapsto\omega(t,\xi)$ is bounded and measurable on $[\tau',\tau'']\times (I_1\cup I_2)$.
Hence by the Lebesgue differentiation theorem there is a set of times ${\mathcal A}\subseteq [\tau',\tau'']$, 
with $meas(\A)=\tau''-\tau'$, such that 
\begi
\item
{\it every $\tau\in \A$ is a Lebesgue point of the map $t\mapsto \omega(t,\xi)$, for every $\xi\in (I_1\cup I_2)\setminus
{\mathcal N}_\tau$, where ${\mathcal N}_\tau$ is a set of measure zero, possibly depending on $\tau$.}
\endi
Next, assume by contradiction that at some time $\tau\in \A$ 
 the curvature
$\omega(t,\xi)$ is not a.e.~constant on the set 
$$\mathcal B_\tau\,\doteq\,\bigl\{\xi \in {\rm Int}(I_1\cup I_2) \,;~ \beta(\tau,\xi)>0\bigr\}.$$
%
%
We can thus choose
 $\xi_1,\xi_2\in\mathcal B_\tau$ with $\omega(\tau,\xi_1)\neq\omega(\tau,\xi_2)$ and such that $\tau$ is a Lebesgue point for $\omega(\cdot,\xi_i)$, $i=1,2$.
 Recalling  that the map $(t,\xi)\mapsto\beta(t,\xi)$ is continuous,
for any $\delta>0$ sufficiently small, using (\ref{intom}) with $[\tau',\tau'']$ replaced by the smaller interval $[\tau-\delta,\tau+\delta]$ where both $\beta(\cdot,\xi_i)>0$, we deduce
$$\int_{\tau-\delta}^{\tau+\delta} \omega(t,\xi_1) \, dt ~=~\int_{\tau-\delta}^{\tau+\delta} \omega(t,\xi_2)\, dt\,.$$
Since $\tau$ is a Lebesgue point of $\omega(\cdot,\xi_i)$, letting $\delta\to 0$ we obtain
$$\omega(\tau,\xi_1)~=~
\lim_{\delta\to 0}\,{1\over 2\delta}\int_{\tau-\delta}^{\tau+\delta} \omega(t,\xi_1) \, dt ~=~ 
\lim_{\delta\to 0}\,{1\over 2\delta}\int_{\tau-\delta}^{\tau+\delta} \omega(t,\xi_2) \, dt~=~\omega(\tau,\xi_2)\,,$$
reaching a a contradiction. \endproof

\section{Optimality conditions at junctions}
\label{sec:5}
\setcounter{equation}{0}
The proof of 
Theorem~\ref{t:21} established the existence of optimal slicing strategies  $\Omega(\cdot)$ within a class of 
functions with BV regularity.  On the other hand, the necessary conditions for optimality 
proved in Theorem~\ref{t:41} require the sets $\Omega(t)$ to have $\C^{1,1}$  boundary.
Indeed, this assumption is needed to uniquely define the perpendicular curves $t\mapsto x(t,\xi)$.

The aim of this section is to partially bridge this regularity gap, ruling out certain configurations 
where the sets $\Omega(t)$ have corners. We consider two cases.

\begi
\item[{\bf (A1)}] {\bf (Two moving arcs)} {\it
There exists $\tau,\delta_0>0$ such that, for $t\in [\tau-\delta_0, \tau+\delta_0]$,  the boundary $\partial\Omega(t)$ contains two adjacent arcs  joining at a point $P(t)$ 
at an angle $\theta(t)$.    These arcs admits  $\C^{1,1}$  parameterizations of the form
$$\left\{ \bega{l}(t,\xi)\mapsto x_1(t,s),\quad s\in [-s_0, 0],~~|t-\tau|\leq \delta_0\,,\\[2mm]
(t,s)\mapsto x_2(t,s),\quad s\in [0,s_0],~~~|t-\tau|\leq \delta_0\,,\enda
\right.$$
with
$$x_1(t,0) = x_2(t,0)
=P(t),\qquad\qquad \bigl| x_{1,s}(\tau,0)\bigr|>0,\quad \bigl| x_{2,s}(\tau,0)\bigr|>0.$$
Moreover, calling $\bfn_i(t,s)$ the unit outer normal to the boundary $\partial \Omega(t)$ at the point $x_i(t,s)$,
one has
$$\la  x_{i,t}(t,s),\, \bfn_i(t,s)\ra~\geq~c_0~>~0\qquad\qquad\forall t,s.$$
}\endi

\begi
\item[{\bf (A2)}] {\bf (A moving arc and a stationary arc)} {\it There exists $\tau,\delta_0>0$ such that, for $t\in [\tau-\delta_0, \tau+\delta_0]$,  the boundary $\partial\Omega(t)$ contains two adjacent arcs  joining at a point $P(t)$ 
at an angle $\theta(t)$.    These arcs admits  $\C^{1,1}$  parameterizations of the form
$$\left\{ \bega{l}(t,s)\mapsto x_1(t,s),\quad s\in [-s_0, 0],~~|t-\tau|\leq \delta_0\,,\\[2mm]
s\mapsto x_2(s),\qquad\quad ~ s\in [0,s_0],\enda
\right.$$
\bel{par2}\qquad\qquad  x_1(t,0) = x_2(s(t))
=P(t),\qquad   \bigl| x_{1,s}(\tau,0)\bigr|>0,\quad \bigl| x_{2,s}(0)\bigr|>0.\eeq 
Calling $\bfn_1(t,s),\bfn_2(s)$ the unit outer normal to the boundary $\partial \Omega(t)$ at the points $x_i(t,s)$, $x_2(s)$, respectively, one has
\bel{move2} \la x_{1,t}(t,s),\, \bfn_1(t,s)\ra~\geq~c_0~>~0\qquad\qquad\forall t,s.\eeq
}\endi

\begin{figure}[ht]
\centerline{\hbox{\includegraphics[width=14cm]{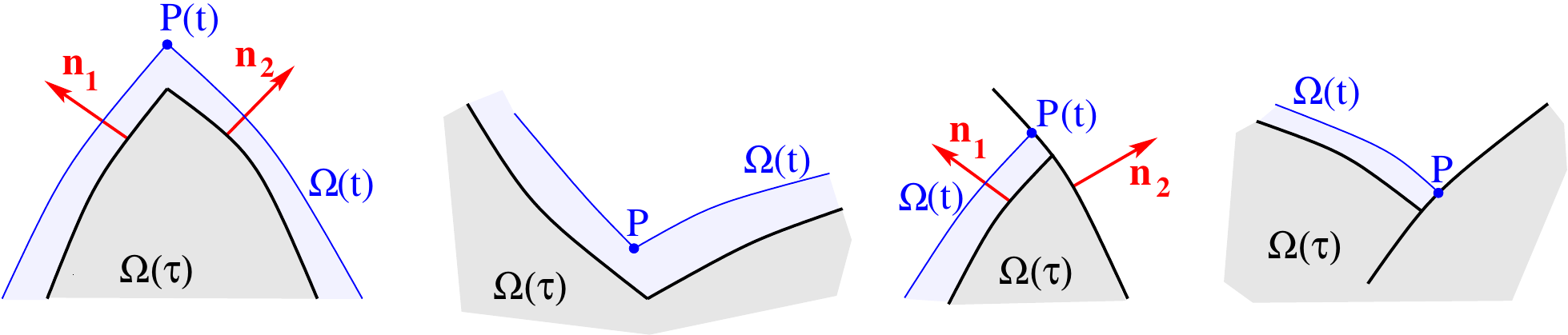}}}
\caption{\small Left two figures: the setting described in {\bf (A1)}, where two moving arcs join at a point $P(t)$, forming an outward or an inward corner.  Right two figures: the setting described in {\bf (A2)}, 
where a moving arc joins a stationary arc, forming a corner at the point $P(t)$. }
\label{f:csm24}
\end{figure}

The next result shows that, for an optimal slicing,
the boundaries of the sets $\partial\Omega(t)$ cannot have corners if at least
one of the two adjacent arcs is moving.
%

\begin{theorem}\label{t:51} 
Let $t\mapsto \Omega(t)$ be an optimal slicing of the set $V$.
In the settings described at {\bf (A1)} or {\bf (A2)},
 the  two arcs $s\mapsto x_i(\tau,s)$, $i=1,2$, must be tangent at $P(\tau)$.
\end{theorem}

\begin{figure}[ht]
\centerline{\hbox{\includegraphics[width=10cm]{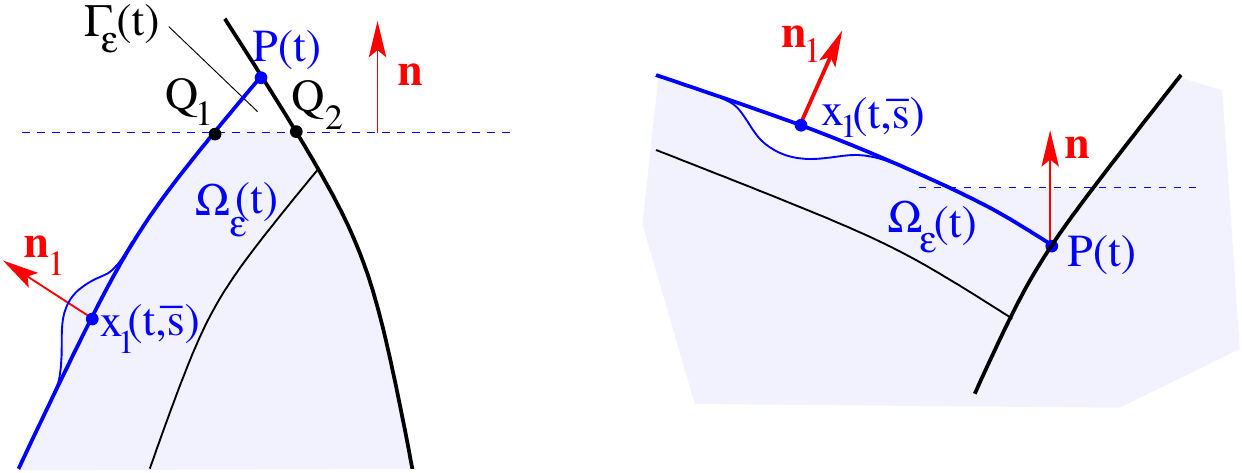}}}
\caption{\small Left: in the case of an outward corner, a better slicing strategy is obtained by
removing a triangle near the point $P(t)$ and adding a region with the same area in a neighborhood of $x_1(t,\bar s)$.   Right: in the case of an inward corner, a better slicing strategy is obtained by
adding a triangle near the point $P(t)$ and removing a region with the same area in a neighborhood of $x_1(t,\bar s)$. }
\label{f:csm25}
\end{figure}

{\bf Proof.} We shall give a proof in the case where the assumptions {\bf (A2)} are satisfied. The analysis of case {\bf (A1)}
is entirely similar.  
By the last two inequalities in (\ref{par2}), by a rescaling of $s$-variables we can assume 
\bel{par22} \bigl| x_{1,s}(\tau,s)\bigr|=1,\qquad \bigl| x_{2,s}(s)\bigr|=1\qquad\quad\forall s.\eeq 
Define the unit vector
\bel{bn}\bfn~\doteq~{\bfn_1(\tau,0) + \bfn_2(0)\over \bigr|\bfn_1(\tau,0) + \bfn_2(0)\bigl|}\,.
\eeq
For $\ve>0$ small enough, introduce the function
\bel{psidef}\psi_\ve(s)~\doteq~\max\{ \ve- |s|, \,0\}\eeq
and let $\vp_\ve$ be as in (\ref{presc}).
Moreover, fix a point $-s_0<\bar s<0$.

We analyze in detail the case where $\Omega(\tau)$ has an outward corner at
$P(\tau)$, where, to get a better strategy, we remove from $\Omega(t)$
a triangular region near the  corner (see Fig.~\ref{f:csm25}, left).
The case where $\Omega(\tau)$ has an inward corner at
$P(\tau)$ can be treated by adding to $\Omega(t)$ a triangular region near the corner
(see Fig.~\ref{f:csm25}, right).
 We omit the proof of this case,
being the computation entirely similar to the outward corner case. 
\v

As shown in Fig.~\ref{f:csm25}, left,  for $\ve>0$ small we
construct a slicing $\Omega_\ve(\cdot)$ of $V$ with lower cost
in the following way.

As a first step, consider the sets $\Hat \Omega_\ve(t)$ obtained by removing from $\Omega(t)$
a triangular region near the outward corner:

\bel{Hom1} \Hat\Omega_\ve(t)~\doteq~\Omega(t)\setminus \Gamma_\ve(t),\eeq
where $\Gamma_\ve(t)$ is the triangular region bounded by the curves $x_1(t,\cdot)$, $x_2(\cdot)$
and the straight line
\bel{Gep1} \Big\{
x\in \R^2\,;~~\la P(t)-x, \bfn\ra~=~\ve^\gamma \psi_\ve(t-\tau)\Big\},\eeq
for some exponent  $\gamma>1$.
The sets $\Hat\Omega_\ve(t)$ now have area $\leq t$.
To achieve the identity
\bel{areat}\caL^2\bigl(\Omega_\ve(t)\bigr)\,=\,t\qquad\qquad\forall t\in [0,T],\eeq
we enlarge each set $\Hat\Omega_\ve(t)$ by pushing  its boundary outward in a neighborhood of the 
point $x(t,\bar s)$.   More precisely, we define $\Omega_\ve(t)$ to be the set obtained from $\Hat \Omega_\ve(t)$ replacing the portion of the boundary $\bigl\{x_1(t, s)\,;~s\in [-s_0, 0]\bigr\}$
with 
\bel{x1ep} x_1^\ve(t,s)~=~x_1(t,s) + \sigma_\ve(t) \vp_\ve(s-\bar s) \bfn_1(t,s),\qquad\qquad
s\in [-s_0, 0],\eeq
for some $\sigma_\ve(t)\geq 0$.  Here $\vp_\ve$ is the function defined at (\ref{presc}), see Fig.~\ref{f:sm76}.
We claim that, for all $\ve>0$ small enough, the following holds.
\begi
\item[(i)] The function $t\mapsto \sigma_\ve(t)$ is uniquely determined by the requirement (\ref{areat}).
\item[(ii)] The multifunction $t\mapsto \Omega_\ve(t)$ is a slicing of $V$.
\item[(iii)] The cost of $\Omega_\ve(\cdot)$ is strictly smaller than the cost of the original strategy $\Omega(\cdot)$.
\endi
\v
The above claims (i)--(iii) will be proved in  several steps.
\v
{\bf 1.} Since $\Omega(\tau)$ has an outward angle at $P(\tau)$, the tangent vectors
$${\bf t}_1~\doteq~x_{1,s}(\tau,0)\,,\qquad\qquad {\bf t}_2~\doteq~x_{2,s}(0)\,,$$
which are unit vectors by (\ref{par22}), are distinct. Set
$$c_1~\doteq~ \langle {\bf t}_1,\bfn\rangle>0\,,\qquad\qquad c_2~\doteq~-\langle {\bf t}_2,\bfn\rangle>0\,,$$
with $\bfn$ as in (\ref{bn}). 
By the continuity of the maps $t\mapsto x_{1,s}(t,0)$ and $t\mapsto x_{2,s}(s(t))$, since $c_1,c_2>0$ at $t=\tau$,
there exist $\delta_1>0$ such that
\bel{c12t}
c_1(t)~\doteq~\langle x_{1,s}(t,0),\bfn\rangle~>~0\,,\qquad c_2(t)~\doteq~-\langle x_{2,s}(s(t)),\bfn\rangle~>~0,
\eeq
 for all $t\in[\tau-\delta_1,\tau+\delta_1]$. In the following we shall
take $\ve<\delta_1$.
\v
{\bf 2.} Together with the point $P(t)$ we let $Q_1(t)$ and $Q_2(t)$ be the two lower vertices of the
triangular region $\Gamma_\ve(t)$ (see Fig.~\ref{f:csm25}, left).  These are the points where the curves 
$x_1(t,\cdot)$ and $x_2(\cdot)$ intersect the line (\ref{Gep1}).   Set
$$\ell(t)\,\doteq\, \ve^\gamma\psi_\ve(t-\tau).$$
Note that the definition (\ref{psidef}) implies  $\ell(t)\leq \ve^{\gamma+1}$ and
$\ell(t)=0$ for $|t-\tau|\ge \ve$.

We now have $Q_1(t) = x_1\bigl(t,\rho_1(t)\bigr)$, $Q_2(t)= x_2 \bigl( s(\rho_2(t)\bigr)$, where
$\rho_1,\rho_2$ are determined by the equations
\bel{r12eq}\la P(t)-x_1(t,-\rho_1(t)),\,\bfn\ra\,=\,\ell(t)\,,\qquad\quad 
\la P(t)-x_2(s(\rho_2(t))),\,\bfn\ra\,=\,\ell(t)\,.
\eeq

By the $\C^{1,1}$ regularity assumed in {\bf (A2)}, the maps $s\mapsto x_{1,s}(t,s)$ and
$s\mapsto x_{2,s}(s)$ are Lipschitz continuous. By (\ref{c12t}), an application of the implicit function theorem
yields the local existence of unique values $\rho_1(t),\rho_2(t)$ which satisfy (\ref{r12eq}). Moreover
\bel{rho12}
\rho_1(t)\,=\,\frac{\ell(t)}{c_1(t)}+\O(1)\cdot \ell^2(t),
\qquad\quad
\rho_2(t)\,=\,\frac{\ell(t)}{c_2(t)}+\O(1)\cdot \ell^2(t).
\eeq
%
%
%
%
\v
\noindent{\bf 3.} In this step we estimate the change in area and the change in perimeter, determined by the removal
of the triangular region $\Gamma_\ve(t)$.

{}From the above construction one immediately obtains an upper bound on the area:
\bel{areaGamma}
\caL^2\big(\Gamma_\ve(t)\bigr)~\leq~C_0\cdot \ell^2(t) ~=~C_0 \ve^{2\gamma} \bigl(\ve -|t-\tau|\bigr)^2
\qquad\forall~~ t\in[\tau-\ve,\tau+\ve],
\eeq
for some constant $C_0$.

To estimate the decrease in the perimeter we observe that
$$Q_1(t)-P(t)=-\rho_1(t)x_{1,s}(t,0)+\O(1)\cdot \rho_1^2(t)\,,\qquad \quad Q_2(t)-P(t)=\rho_2(t)x_{2,s}(s(t))+\O(1)\cdot \rho_2^2(t),$$
$$Q_2(t)-Q_1(t)~=~\rho_1(t)x_{1,s}(t,0)+\rho_2(t)x_{2,s}(s(t))+\O(1)\cdot \bigl(\rho_1^2(t)+\rho_2^2(t)\bigr).$$
Since the two arcs removed have total length $\rho_1(t)+\rho_2(t)$, while the new segment  has length
$\bigl|Q_2(t)-Q_1(t)\bigr|$, the net change in boundary length is
\bel{DLraw}
\H^1\bigl(\partial\Hat\Omega_\ve(t)\bigr)-\H^1\bigl(\partial\Omega(t)\bigr)
\,=\,\Big|\rho_1(t)x_{1,s}(t,0)+\rho_2(t)x_{2,s}(s(t))\Big|-\bigl(\rho_1(t)+\rho_2(t)\bigr)
+\O(1)\cdot\bigl(\rho_1^2(t)+\rho_2^2(t)\bigr).
\eeq
Since the two unit vectors $x_{1,s}(t,0)$ and $x_{2,s}(s(t))$ are not parallel, and make a uniformly positive
angle for all $t$ sufficiently close to $\tau$, we conclude
\bel{DL}
\H^1\bigl(\partial\Hat\Omega_\ve(t)\bigr)-\H^1\bigl(\partial\Omega(t)\bigr)
~\leq\,-\,\kappa_0\,\ell(t)~=~-\kappa_0 \ve^\gamma \bigl(\ve-|t-\tau|\bigr)\qquad\forall t\in [\tau-\ve, \tau+\ve],
\eeq
for some constant $\kappa_0>0$.

Integrating (\ref{DL}) for $t\in[\tau-\ve,\tau+\ve]$  we
obtain
\bel{gaincorner}
\int_{\tau-\ve}^{\tau+\ve}\Bigl[\H^1\bigl(\partial\Hat\Omega_\ve(t)\bigr)
-\H^1\bigl(\partial\Omega(t)\bigr)\Bigr]dt
~\le\,-\kappa_0\,\ve^{\gamma+2}\,,
\eeq
for all $\ve>0$ sufficiently small.
\v
{\bf 4.} We are now ready to prove (i), showing that $\sigma_\ve(t)$ is uniquely determined.
Consider the area function
$$\sigma~\longmapsto~A_\ve(\sigma,t)~\doteq~\caL^2\Big(\Hat\Omega_\ve(t)\ \cup\ 
\bigl\{x_1(t,s)+\zeta\vp_\ve(s-\bar s)\bfn_1(t,s)\,;~s<0, ~\zeta\in [0, \sigma]\bigr\}\Big).$$
where $\bfn_1(t,s)$ denotes the unit outer normal at $x_1(t,s)$ and $\vp_\ve$ is the function at (\ref{presc}).
By (\ref{vpr1}), since the curve $s\mapsto x_1(s,t)$ is parametrized by arc-length, recalling (\ref{areaGamma})
we obtain
\bel{Aep}A_\ve(0,t)\,=\,\caL^2\bigl(\Hat\Omega_\ve(t)\bigr)\,\in\,\Big[ t -C_0 \ve^{2\gamma} \bigl(\ve -|t-\tau|\bigr)^2, ~ t\Big],\qquad\qquad {d\over d\sigma} A_\ve(\sigma,t)\bigg|_{\sigma=0}~=~{\ve\over 2}\,.\eeq
By the $\C^{1,1}$ regularity assumed in {\bf (A2)}, for all $\ve>0$ sufficiently small the implicit function theorem
yields a unique  $\sigma_\ve(t)$ such that 
\bel{asip}A_\ve\bigl(\sigma_\ve(t),t\bigr)~=~t\,.\eeq
We observe that, by (\ref{Aep}),
\bel{sigorder}
\sigma_\ve(t)~=~\O(1)\cdot \ve^{2\gamma+1}\qquad \hbox{for}\quad t\in[\tau-\ve,\tau+\ve]\,,
\eeq
while $\sigma_\ve(t)=0$ for $t\notin[\tau-\ve,\tau+\ve]$.
\v
\noindent{\bf 5.}  To estimate the increase in the perimeter caused by the deformation (\ref{x1ep}),
we use the same computation of (\ref{xxi1})--(\ref{xxe}) in the proof of Theorem~\ref{t:41}. 
Since  
$$\int\vp_\ve(s)\, ds\,=\,{\ve\over 2}\,,\qquad\int\bigl[(\vp_\ve'(s))^2+\vp^2_\ve(s)\bigr]\, ds \,=\,\O(1)\cdot\ve^{-1}\,,$$
by (\ref{sigorder}) it follows
\[
\bega{l}\ds\int_{-s_0}^0\Bigl(\bigl|x_{1,s}^\ve(t,s)\bigr|-\bigl|x_{1,s}(t,s)\bigr|\Bigr)ds\\[4mm]
\qquad\qquad\ds=\label{compest2}~-\,\omega_1(t,\bar s)\,\sigma_\ve(t)\int\vp_\ve(s-\bar s)\,ds
+\O(1)\cdot\sigma_\ve(t)^2\int\bigl[(\vp_\ve'(s-\bar s))^2+(\vp_\ve(s-\bar s))^2\bigr]ds\\[4mm]
\ds\qquad\qquad=~\O(1)\cdot \ve^{2\gamma+2}+\O(1)\cdot \ve^{4\gamma+1}.\enda
\]
Integrating over $t\in[\tau-\ve,\tau+\ve]$ we find
\bel{compcost}
\int_{\tau-\ve}^{\tau+\ve}\Bigl[\H^1\bigl(\partial\Omega_\ve(t)\bigr)
-\H^1\bigl(\partial\Hat\Omega_\ve(t)\bigr)\Bigr]dt
~=~\O(1)\cdot\ve^{2\gamma+3}+\O(1)\cdot\ve^{4\gamma+2}.
\eeq
Since $\gamma>1$, combining
(\ref{gaincorner}) with (\ref{compcost}) one obtains
\bel{netgain}\int_0^T\Bigl[\H^1\bigl(\partial\Omega_\ve(t)\bigr)
-\H^1\bigl(\partial\Omega(t)\bigr)\Bigr]dt
\,=\,
\int_{\tau-\ve}^{\tau+\ve}\Bigl[\H^1\bigl(\partial\Omega_\ve(t)\bigr)
-\H^1\bigl(\partial\Omega(t)\bigr)\Bigr]dt
\,\le\,-\frac{\kappa_0}{2}\,\ve^{\gamma+2}\,<\,0
\eeq
for all $\ve>0$ sufficiently small. This proves (iii).
\v
{\bf 6.} It remains to prove (ii), showing that  for all $\ve>0$ sufficiently 
small the modified map $t\mapsto \Omega_\ve(t)$ is still a slicing of $V$.

The construction performed in step {\bf 4} guarantees that $\caL^2(\Omega_\ve(t))=t$ for every $t$. 
We claim that the
monotonicity condition (\ref{adm2}) also holds.   
Namely, every point along the boundary $\partial\Omega_\ve(t)$ moves
outward as $t$ increases.  Compared with $\partial \Omega(t)$,
two portions of the boundary have been modified: the chord
$\Hat{Q_1(t),Q_2(t)}$ cutting away the corner, and the bump (\ref{x1ep}) around $x_1(t,\bar s)$.
To handle both cases, we shall use the crucial assumption (\ref{move2}) that the boundary points  $x_1(t,s)$
move strictly outward for $t\approx \tau$ and $s\approx \bar s$.

{\bf (1) The chord.}    By the assumption (\ref{move2}) it follows that 
$${d\over dt} \la P(t), \bfn\ra~>~\kappa_1~>0$$
for some constant $\kappa_1$ and all times $t$ in a neighborhood of $\tau$.
By (\ref{r12eq}), for $i=1,2$ and all $\ve>0$ small enough,  we have
$${d\over dt}  \la Q_i(t), \bfn\ra ~=~{d\over dt} \la P(t), \bfn\ra - {d\over dt} \ell(t)~>~ \kappa_1 + \O(1) \cdot \ve^\gamma~\geq~{\kappa_1\over 2}\,,$$
showing that all points along the chord $\Hat{Q_1(t),Q_2(t)}$  move outward with speed $\geq \kappa_1/2$.

{\bf (2) The bump.} 
Here the analysis is essentially the same as in step {\bf 6}, CASE 1, of the  proof of
Theorem~\ref{t:41}.

Calling $\bfn_1^\ve(t,s)$ the unit outer normal to $\Omega_\ve(t)$ at the boundary point 
$x_1^\ve(t,s)$    defined at    (\ref{x1ep}),
by (\ref{xtn}) we have
\bel{xtn1} \langle x^\ve_t, \bfn_1^\ve\rangle~=~ \langle x^\ve_t, \bfn_1\rangle 
+  \langle x^\ve_t, \bfn_1^\ve-\bfn_1\rangle~=~ \langle x_t, \bfn_1\rangle 
+\left({d\over dt} \sigma_\ve\right) \vp_\ve + \la x^\ve_t, \bfn_1^\ve-\bfn_1\ra~>~0.
\eeq
Indeed, by  (\ref{move2}) the  term $ \langle x_t, \bfn_1\rangle$ on the right hand side remains uniformly positive for all $(t,s)$ in a neighborhood
of $(\tau,\bar s)$.  On the other hand, the last two terms can be rendered arbitrarily small by choosing $\gamma>1$
and taking $\ve>0$ small enough.
This establishes (ii), completing the proof of the theorem.
\endproof

\section{Optimality conditions at boundary points}
\label{sec:6}
\setcounter{equation}{0}
In this section we study the behavior of the boundary $\partial \Omega(\tau)\cap V$ at a point $P$
where it
meets the boundary $\partial V$. More precisely, the following situation will be considered.
\begi
\item[{\bf (A3)}] {\it
There exists $\tau,\delta_0>0$ such that, for $s\in [-\delta_0, \delta_0]$,
the map $s\mapsto y(s)$ is a $\C^1$ arc-length parameterization of a portion of the boundary $\partial V$.
 
Moreover, a portion of boundary $\partial\Omega(t)$ admits a $\C^{1,1}$ parameterization of the form
$$(t,s)\mapsto x(t,s),\quad s\in [0, s_0],~~|t-\tau|\leq \delta_0\,,$$
such that, for some nondecreasing function $s(\cdot)$,
\bel{par5}x(t,0) = y(s(t)),\qquad x(\tau,0)=y(0)=P,\qquad \bigl| x_s(\tau,0)\bigr|>0.\eeq
 
}\endi
We denote by
$$
\bfe~\doteq~{d\over ds} y(s)\bigg|_{s=0}$$
the unit vector tangent to the boundary $\partial V$ at the point $P=y(0)$.
Notice that $x_s(\tau,0)$ is a tangent vector to the boundary $\partial \Omega(\tau)$ at the same point
$P$. In the above setting, we have
  
\begin{figure}[ht]
\centerline{\hbox{\includegraphics[width=11cm]{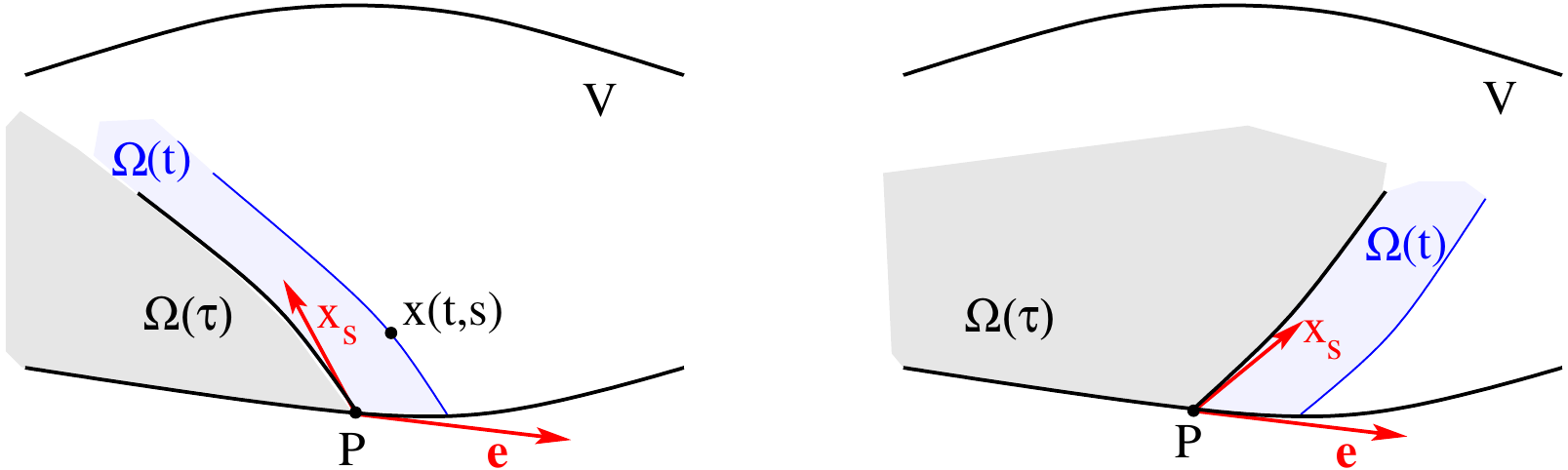}}}
\caption{\small Two cases where the boundary $\partial \Omega(\tau)\cap V$ intersects the boundary 
of $V$ at a point $P= y(0)$, in a non-perpendicular way.}
\label{f:csm26}
\end{figure}

\begin{theorem} \label{t:61} Let $t\mapsto\Omega(t)$ be an optimal slicing strategy for the
set $V$, and let the assumptions {\bf (A3)} hold.
Then
\begi
\item[(i)] either the junction is perpendicular:
\bel{prp} \la \bfe, \, x_s(\tau,0)\ra~=~0,\eeq
\item[(ii)] or else the junction point does not move:
\bel{still} x_t(\tau,0)\,=\,0.\eeq
\endi
\end{theorem}
 
{\bf Proof.} By the last assumption in (\ref{par5}), performing a variable change we can assume
that at time $t=\tau$ the derivative of the map $s\mapsto x(\tau,s)$ satisfies
$$\bigl|x_s(\tau,s)\bigr|~=~1\qquad\quad\forall s.$$
If (\ref{still}) fails, by possibly shrinking the values of $s_0, \delta_0$, by continuity we can assume
$x(t,0)=y\bigl(s(t)\bigr) $ with
\bel{move5}
{d\over dt} s(t)\,\geq\,c_0\,>\,0\qquad
\forall t\in[\tau-\delta_0,\tau+\delta_0],\eeq
\bel{move6} \la \bfn(t,s),\,x_t(t,s)\ra~\geq~c_1~>~0\qquad
\hbox{for}~~|t-\tau|\leq \delta_0, ~ s\in [0, s_0],\eeq
where $\bfn(t,s)$ denotes the unit outer normal to $\partial\Omega(t)$ at $x(t,s)$.
\v
Assume that the perpendicularity relation (\ref{prp}) also fails. In the following we work out a proof assuming that
\bel{acute} \la \bfe, \, x_s(\tau,0)\ra~<~0.\eeq
In this case, we start by constructing the sets $\Hat\Omega_\ve(t)$ obtained by removing from $\Omega(t)$ 
a triangular region $\Gamma_\ve(t)$ near the outward corner (see Fig.\ref{f:csm27}, left) and eventually reach a contradicion. \\
In the case where
\bel{obtuse} \la \bfe, \, x_s(\tau,0)\ra~>~0.\eeq
we start by adding to $\Omega(t)$ a triangular region $\Gamma_\ve(t)$ near the inward corner (see Fig.~\ref{f:csm27}, right) and reach a contradiction
by an entirely similar argument.

\begin{figure}[ht]
\centerline{\hbox{\includegraphics[width=12cm]{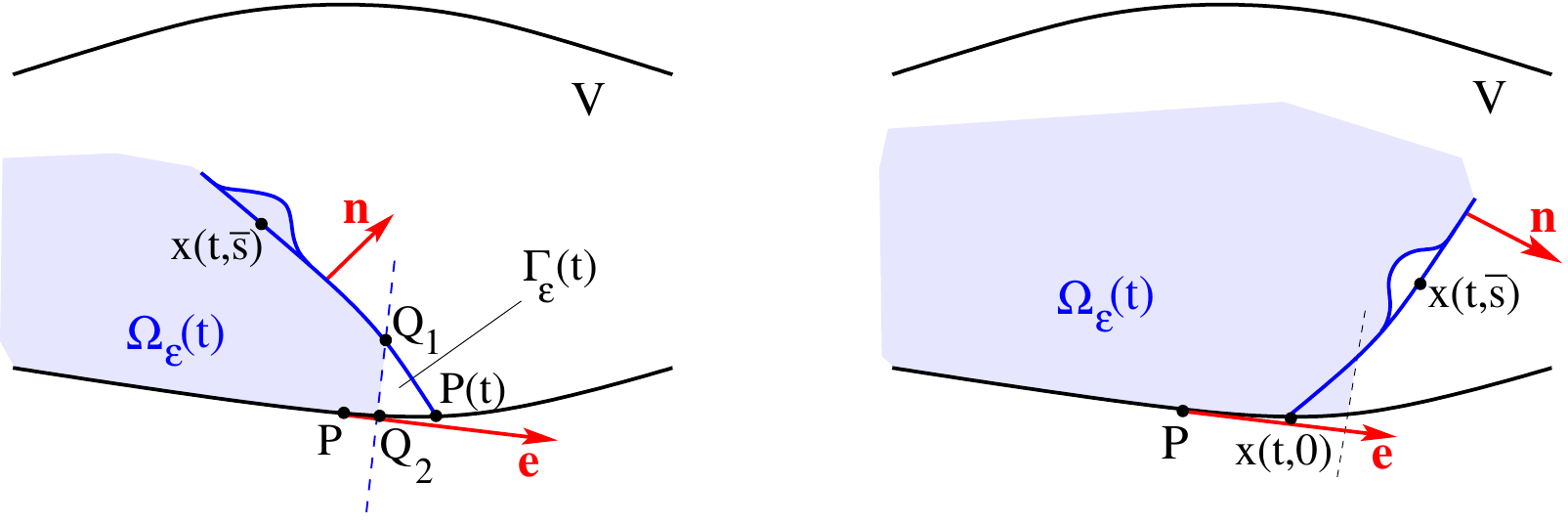}}}
\caption{\small Left: the construction of the slicing strategy $\Omega_\ve(\cdot)$ in the case (\ref{acute}). Right:  the slicing strategy $\Omega_\ve(\cdot)$ in the case (\ref{obtuse}).}
\label{f:csm27}
\end{figure}

With reference to Fig.~\ref{f:csm27}, left, let $P(t)$ be the point where the relative boundary of $\Omega(t)$ meets the 
boundary of $V$, so that $P(\tau)=P$ is the point in (\ref{par5}). 
Let $\Gamma_\ve(t)$ be the triangular region enclosed between $\partial V$, $\partial \Omega(t)$, and the
line
\bel{Gep17} \Big\{
x\in \R^2\,;~~\la P(t)-x, \bfe\ra~=~\ve^\gamma \psi_\ve(t-\tau)\Big\},\eeq
where $\psi_\ve$ is the function introduced at (\ref{psidef}).  Finally, define
 $$\Hat\Omega_\ve(t)\,\doteq \,\Omega(t)\setminus \Gamma_\ve(t).$$ 
 Clearly, the sets $\Hat\Omega_\ve(t)$ now have area $\leq t$.
To achieve the identity (\ref{areat}), we fix $0<\bar s<s_0$ and
enlarge each set $\Hat\Omega_\ve(t)$ by pushing its boundary outward in a neighborhood of the
point $x(t,\bar s)$.
 
More precisely, calling $\bfn(t,s)$ the unit outer normal to the boundary $\partial \Omega(t)$
at the point $x(t,s)$, we define $\Omega_\ve(t)$ to be the set obtained from $\Hat \Omega_\ve(t)$ replacing the portion of the boundary $\bigl\{x(t, s)\,;~s\in [0, s_0]\bigr\}$
with
\bel{xepp} x^\ve(t,s)~=~x(t,s) + \sigma_\ve(t) \vp_\ve(s-\bar s) \bfn(t,s),\qquad\qquad
s\in [0, s_0],\eeq
for some $\sigma_\ve(t)\geq 0$. As in the proof of Theorem~\ref{t:41},
we claim that, for all $\ve>0$ small enough, the following holds.
\begi
\item[(i)] The function $t\mapsto \sigma_\ve(t)$ is uniquely determined by the area identity (\ref{areat}).
\item[(ii)] The multifunction $t\mapsto \Omega_\ve(t)$ is a slicing of $V$.
\item[(iii)] The cost of $\Omega_\ve(\cdot)$ is strictly small than the cost of the original strategy $\Omega(\cdot)$.
\endi
 The arguments to prove (i)--(iii) are essentially the same as in the proof of Theorem~\ref{t:51}.
The only difference is that now the
cost functional does not account for the portion of $\partial\Omega(t)$ that lies on $\partial V$. 
More precisely, consider the points $Q_1(t)$, $Q_2(t)$  such that 
\bel{Q12d}\la P(t)-Q_i(t), \bfe\ra~=~\ell(t)~\doteq~\ve^\gamma \psi_\ve(t-\tau), \qquad\qquad i=1,2,\eeq
and $Q_1(t)\in \partial \Omega(t)\cap V$, ~$Q_2(t)\in \partial V$, respectively (see Fig.~\ref{f:csm27}, left).
Then
\bel{costep}\bega{rl}
\H^1\bigl( \partial \Hat \Omega_\ve(t)\cap V\bigr)- \H^1\bigl( \partial \Hat \Omega(t)\cap V\bigr)
&=~
[\hbox{length of the segment with endpoints $Q_1(t), Q_2(t)$} ]\\[2mm]
&\qquad -~ [\hbox{length of the arc with endpoints $Q_1(t), P(t)$]}\\[2mm]
&\leq~-\kappa_0 \ell(t)
\enda\eeq
for some constant $\kappa_0>0$ and all $\ve>0$ small enough.
Indeed, both arcs $\Hat{Q_1(t) Q_2(t)}$ and $\Hat{Q_1(t) P(t)}$ have the same order of magnitude as $\ell(t)$.
Moreover, for $t$ close to $\tau$, the chord $\Hat{Q_1 Q_2}$ meets 
the boundary $\partial V$ at an almost perpendicular angle, while by the assumption (\ref{acute}) the arc $\Hat{Q_1 P}$ meets $\partial V$ at an angle bounded away from 
$\pi/2$.  

The claims (i)--(iii) will  be proved in a few steps, following those  in Section~\ref{sec:5}.
 \v
\noindent{\bf 1.} The points $Q_1, Q_2$ are defined more precisely as
$Q_1(t) = x\bigl(t,\rho_1(t)\bigr)$,  $Q_2(t)= y \bigl( s(\rho_2(t)\bigr)$,  where 
$\rho_1,\rho_2$ are determined by the equations
\bel{r12eq}\la P(t)-x(t,\rho_1(t)),\,\bfe\ra\,=\,\ell(t)\,,\qquad\quad 
\la P(t)-y(s(\rho_2(t))),\,\bfe\ra\,=\,\ell(t)\,.
\eeq
The upper bound (\ref{areaGamma}) on the area of the triangular region $\Gamma_\ve(t)$ 
remains valid, as well as the estimate (\ref{gaincorner}) on the decrease of the perimeter.
\v
{\bf 2.}
Defining the area function
$$\sigma~\mapsto~A_\ve(\sigma,t)~\doteq~\caL^2\Big(\Hat\Omega_\ve(t)\ \cup\ 
\bigl\{x(t,s)+\zeta\vp_\ve(s-\bar s)\bfn(t,s)\,;~s<0, ~\zeta\in [0, \sigma]\bigr\}\Big),$$
the identities in (\ref{Aep}) still hold.  By the implicit function theorem, for all $\ve>0$ small enough there exists a unique
$\sigma_\ve(t)$ such that (\ref{asip}) and (\ref{sigorder}) hold. This proves (i).
\v
{\bf 3.}
The increase in the perimeter caused by the deformation of the boundary at (\ref{xepp}) satisfies the same bounds
as in
(\ref{compcost}).    Combining (\ref{gaincorner}) with (\ref{compcost}) we obtain (\ref{netgain}), proving (iii).
\v
{\bf 4.}  It remains to prove (ii), showing that  for all $\ve>0$ sufficiently 
small the modified map $t\mapsto \Omega_\ve(t)$ is still a slicing of $V$.
The  construction of the bump near the point $x(t, \bar s)$ already guarantees that
$\caL^2(\Omega_\ve(t))=t$ for every $t$. 
We claim that the
monotonicity condition (\ref{adm2}) also holds.   
Namely, every point along the boundary $\partial\Omega_\ve(t)$ moves
outward as $t$ increases.  Compared with $\partial \Omega(t)$,
two portions of the boundary have been modified: the chord
$\Hat{Q_1(t),Q_2(t)}$ cutting away the corner, and the bump (\ref{x1ep}) around $x(t,\bar s)$.

{\bf (1) The chord.}    By the assumption (\ref{move5}) it follows
$${d\over dt} \la P(t), \bfe\ra~>~\kappa_1~>0$$
for some constant $\kappa_1$ and all times $t$ in a neighborhood of $\tau$.
By (\ref{r12eq}), for $i=1,2$ and all $\ve>0$ small enough,  we have
$${d\over dt}  \la Q_i(t), \bfe\ra ~=~{d\over dt} \la P(t), \bfe\ra - {d\over dt} \ell(t)~>~ \kappa_1 + \O(1) \cdot \ve^\gamma~\geq~{\kappa_1\over 2}\,,$$
showing that all points along the chord $\Hat{Q_1(t),Q_2(t)}$  move outward with speed $\geq \kappa_1/2$.

{\bf (2) The bump.} 
Calling $\bfn^\ve(t,s)$ the unit outer normal to $\Omega_\ve(t)$ at the boundary point 
$x^\ve(t,s)$    defined at    (\ref{xepp}),
as in  (\ref{xtn}) we have
\bel{xtn1} \langle x^\ve_t, \bfn^\ve\rangle~=~ \langle x^\ve_t, \bfn\rangle 
+  \langle x^\ve_t, \bfn^\ve-\bfn\rangle~=~ \langle x_t, \bfn\rangle 
+\left({d\over dt} \sigma_\ve\right) \vp_\ve + \la x^\ve_t, \bfn^\ve-\bfn\ra~>~0.
\eeq
Indeed, by  (\ref{move6}) the  term $ \langle x_t, \bfn\rangle$ on the right hand side remains uniformly positive for all $(t,s)$ in a neighborhood
of $(\tau,\bar s)$.  On the other hand,  the last two terms can be rendered arbitrarily small by choosing $\gamma>1$
and taking
$\ve>0$ small enough.
This establishes (ii), completing the proof of the theorem.
\endproof

\end{document}